\documentclass[11pt]{article}
\usepackage{amsmath,amsthm,amsfonts,amssymb,bm,wasysym}
\usepackage{epsfig}
\usepackage[usenames]{color}
\usepackage{verbatim}
\usepackage{hyperref}
\usepackage{multicol}
\usepackage{comment}
\usepackage{float}
\usepackage{graphicx}
\usepackage[colorinlistoftodos]{todonotes}
\usepackage[normalem]{ulem}

\usepackage[utf8]{inputenc}

\parskip \medskipamount
\newtheorem{theorem}{Theorem}[section]

\numberwithin{equation}{section}
\newtheorem{lemma}[theorem]{Lemma}
\newtheorem{proposition}[theorem]{Proposition}
\newtheorem{remark}[theorem]{Remark}

\numberwithin{equation}{section}

\def\N{\mathbb{N}}
\def\Z{\mathbb{Z}}

\def\R{\mathbb{R}}

\def\bP{\mathbb{P}}

\def\F{\mathcal{F}}

\def\B{\mathcal{B}}

\def\bP{\mathbb{P}}

\renewcommand{\phi}{\varphi}
\renewcommand{\epsilon}{\varepsilon}

\def\RR{\mathcal{R}}
\allowdisplaybreaks

\newcommand{\1}{{\text{\Large $\mathfrak 1$}}}

\newcommand{\var}{\operatorname{var}}

\renewcommand{\emptyset}{\varnothing}

\newcommand{\til}{\widetilde}
\def\reff#1{(\ref{#1})}

\newcommand{\pr}[1]{\mathbb{P}\!\left(#1\right)}
\newcommand{\E}[1]{\mathbb{E}\!\left[#1\right]}
\newcommand{\estart}[2]{\mathbb{E}_{#2}\!\left[#1\right]}
\newcommand{\prstart}[2]{\mathbb{P}_{#2}\!\left(#1\right)}

\newcommand{\escond}[3]{\mathbb{E}_{#3}\!\left[#1\;\middle\vert\;#2\right]}

\newcommand{\norm}[1]{\left\| #1 \right\|}

\newcommand{\tn}{|\kern-.1em|\kern-0.1em|}

\newcommand{\cp}{\mathrm{Cap}}

\newcommand{\cc}[1]{\mathrm{Cap}\left(#1\right)}

\newcommand\be{\begin{equation}}
\newcommand\ee{\end{equation}}

\def\bP{\mathbb{P}}

\title{\bf Strong law of large numbers for the capacity of the Wiener sausage in dimension four}

\begin{document}
\author{Amine Asselah \thanks{
Universit\'e Paris-Est Cr\'eteil; amine.asselah@u-pec.fr} \and
Bruno Schapira\thanks{Aix-Marseille Universit\'e, CNRS, Centrale Marseille, I2M, UMR 7373, 13453 Marseille, France;  bruno.schapira@univ-amu.fr} \and Perla Sousi\thanks{University of Cambridge, Cambridge, UK;   p.sousi@statslab.cam.ac.uk} 
}
\date{}

\maketitle

\begin{abstract}
We prove a strong law of large numbers for the Newtonian capacity of a 
Wiener sausage in the critical dimension four.
\newline
MSC 2010 \emph{subject classifications.} Primary 60F05, 60G50.
\end{abstract}

\section{Introduction}

We denote by $(\beta_s,s\ge 0)$ a Brownian motion on $\R^4$, and for $r>0$ and $0\le s\le t< \infty$,
the Wiener sausage of radius $r$ in the time period $[s,t]$ is defined as
\begin{equation}
\label{wienersaus}
W_r[s,t]=\{z\in \R^4\ :\ \|z-\beta_u\|\le r \ \text{ for some }s\le u\le  t\}.
\end{equation}
Let $\bP_x$ and $\mathbb E_x$ be the law and expectation with
respect to the Brownian motion started at site $x$,
and let $G$ denote Green's function and $H_A$ denote
the hitting time of $A\subset \R^4$ by the Brownian motion.
The Newtonian capacity of a compact set $A\subset \R^4$ 
may be defined through hitting time as 
\begin{equation}
\label{cap.def}
\cp (A) = \lim_{\|x\|\to \infty}\ \frac{\prstart{H_A<+\infty}{x}}{G(x)}.
\end{equation}
A more classical definition through a variational expression reads
\[
\cp(A)=\Big( \inf\{\int\!\int G(x-y) d\mu(x)d\mu(y):\ \mu \ \
\text{prob. measure with support in $A$}\}\Big)^{-1}.
\]
Our central object is the capacity of the Wiener sausage, and
formula \reff{cap.def}, with $A=W_1[0,t]$, casts the problem
into an intersection event for two independent sausages.

Our main result is the following law of large number for the
capacity of the Wiener sausage. 
\begin{theorem}
\label{theo.cap.wiener}
In dimension four, for any radius $r>0$, almost surely and in $L^p$, 
for any $p\in [1,\infty)$, we have
\begin{equation}
\label{SLLN}
\lim_{t\to \infty} \ \frac{\log t}{t} \, \cp(W_r[0,t])   \ = \ 4\pi^2.
\end{equation}
\end{theorem}

The proof of \eqref{SLLN} presents some similarities with the proof in the discrete case, which is given in our companion 
paper \cite{ASS}, but also substantial differences. 
The main difference concerns the computation of the expected capacity, which in the discrete setting had been essentially obtained by Lawler, see \cite{Asselah:2016hc} for details, whereas in our context it requires new 
delicate analysis.

It may seem odd that the fluctuations result we obtain in 
the discrete model~\cite{ASS} are not directly 
transposable in the continuous setting.
However, it was noticed some thirty years ago
by Le Gall~\cite{LG86} that 
{\it it does not seem easy to deduce Wiener
sausage estimates from random walks estimates}, and vice-versa.
Let us explain one reason for that.
The capacity of a set $A$ can be represented as the integral of the
{\it equilibrium measure} of the set $A$, very much as in the discrete formula
for the capacity of the range $\RR[0,n]$ of a random walk (with
obvious notation)
\[
\cp(\RR[0,n])=\sum_{x\in \RR[0,n]} \prstart{H^+_{\RR[0,n]}=\infty}{x}.
\]
Whereas Lawler \cite{Law91} has established deep non-intersection
results for two random walks in dimension four, the corresponding
results for the equilibrium measure of $W_1(0,t)$ are still missing.

As noted in \cite{ASS}, the scaling in Theorem~\ref{theo.cap.wiener} is analogous to that of the law of large numbers for the volume of the
Wiener sausage in $d=2$ (see~\cite{LGsaucisse}).

\begin{remark}{\em
The limit in \reff{SLLN} is independent of the radius revealing
the scale invariance property, a sign of criticality,
of the scaled limit of the capacity
of the sausage in dimension four. Let us explain better the scaling,
and the criticality of $d=4$. The hitting-time representation
\reff{cap.def} yields a similar formula when we take expectation.
Thus, if $\widetilde W$ denotes a sausage built from $\til \beta$
a Brownian motion independent of $\beta$, then we establish
\begin{equation}\label{exp-cap}
\E{\cp(W_1[0,t])}=\lim_{\|z\|\to\infty}
\frac{1}{G(z)}\prstart
{W_{1/2}[0,t]\cap \widetilde W_{1/2}[0,\infty) \not= \emptyset}{0,z}.
\end{equation}
Now two independent Wiener sausages $W_{1/2}[0,t]$ and
$\widetilde W_{1/2}[0,\infty)$, started at a distance of
order $\sqrt t$, meet with non-vanishing probability (as $t\to\infty$)
only if $d<4$. Indeed, when $\norm{z}$ is of 
order $\sqrt t$, then
\begin{equation}\label{inter-dim}
\prstart
{W_{1/2}[0,t]\cap \widetilde W_{1/2}[0,\infty) \not= \emptyset}{0,z}
 \approx\left\{ \begin{array}{ll}
1 & \textrm{if }d=3\\
\frac{1}{\log t} & \textrm{if }d=4\\
\frac{1}{t^{(d-4)/2}} & \textrm{if }d\ge 5.
\end{array}
\right.
\end{equation}
To make a link with \eqref{exp-cap}, at least at a heuristic level,
recall that $W_{1/2}[0,t]$ lives basically in $\B(0,\sqrt t)$, the
Euclidean ball of center $0$ and radius $\sqrt t$, 
and then we
condition $\til \beta$ on hitting first $\B(0,\sqrt t)$. Note that
by \reff{cap.def} and the scaling property of Brownian motion
\begin{equation}\label{hitting-ball}
\lim_{\|z\|\to\infty}
\frac{1}{G(z)}\prstart{H_{\B(0,\sqrt t)}<\infty}{z}=\cp(\B(0,\sqrt t))=
t^{(d-2)/2} \cp(\B(0,1)).
\end{equation}
Thus, at a heuristic level, combining \reff{inter-dim} and \reff{hitting-ball}, we obtain
\[
\E{\cp(W_1[0,t])}\approx\left\{ \begin{array}{ll}
\sqrt t & \textrm{if }d=3\\
\frac{t}{\log t} & \textrm{if }d=4\\
t & \textrm{if }d\ge 5.
\end{array}
\right.
\]
}
\end{remark}

\begin{remark}{\em
Our result is indeed a result about non-intersection probabilities
for two independent Wiener sausages, and the asymptotic result \reff{SLLN}
reads as follows. For any $\epsilon>0$, almost surely, for $t$ large enough, 
\begin{equation}\label{SLLN-bis}
(1-\epsilon) \frac{2t}{\log t}\le
\lim_{\|z\|\to\infty} \|z\|^2 \cdot \prstart{
W_{1/2}[0,t]\cap \widetilde W_{1/2}[0,\infty)\not= \emptyset\ \Big|\ \beta}{0,z}
\le (1+\epsilon) \frac{2t}{\log t}.
\end{equation}
Estimates, up to constants, have been obtained in a different
regime (where $z$ and $t$ are related as $z=\sqrt t x$) by Pemantle, Peres and Shapiro \cite{PPS}, but cannot be used to obtain our strong
law of large number.}
\end{remark}

One delicate part in Theorem~\ref{theo.cap.wiener} is
establishing convergence for the scaled expected capacity. This
is Proposition~\ref{prop-expected} of Section~\ref{Sec.expcap}.
From \reff{cap.def}, the expected capacity of a Wiener sausage is
equivalent to the
probability that two Wiener sausages intersect. Estimating such
a probability has a long tradition: pioneering works were produced 
by Dvoretzky, Erd\"os and Kakutani \cite{DEK} and Aizenman \cite{A};
Aizenman's results have been subsequently improved by 
Albeverio and Zhou \cite{AZ}, 
Peres \cite{P}, Pemantle, Peres and Shapiro \cite{PPS} and Khoshnevisan~\cite{K} (and references therein). 
In the discrete setting, the literature is even larger and older,
and analogous results are presented in Lawler's 
comprehensive book \cite{Law91}.

As a byproduct of our arguments, 
we improve a Large Deviation estimate of Erhard and Poisat \cite{EP},
and obtain a {\it nearly correct} estimate of the variance,
which will have to be improved for studying the fluctuations.
\begin{proposition}\label{cor-LD}
There is a constant $c>0$, such that for any $0<\epsilon<1$, there
exists $\kappa=\kappa(\epsilon)$ such that for any $t$ large enough
\begin{equation}\label{ineq-LD}
\pr{ \cp(W_1[0,t])-\E{\cp(W_1[0,t])}\ge \epsilon\frac{t}{\log t}}
\le \exp\big(- c\, \epsilon^2 t^{\kappa}\big).
\end{equation}
Moreover, there exists a constant $C>0$, such that for  
$t$ large enough, 
\begin{equation}\label{ineq-var}
\var\big(\cp(W_1[0,t])\big)\ \le\ 
C\, (\log \log t)^9\frac{t^2}{(\log t)^4}.
\end{equation}
\end{proposition}
\begin{remark}{\em We do not know what is the correct
speed in the large deviation estimate \reff{ineq-LD}.
The analogous result for the volume of
the sausage in $d=2$ (or even the size of the range of a random
walk) is not known. On the other hand, the correct order for the variance should be $t^2/(\log t)^4$, as was proved in the discrete setting \cite{ASS}. Thus our bound in \reff{ineq-var}
is off only by a $(\log \log t)^9$ term.}
\end{remark}

One key step of our investigation is a simple formula
for the capacity of the sausage which is neither asymptotic nor
variational. In Section~\ref{subsec.cap}, we deduce a decomposition 
formula for the capacity of the union of two sets in terms
of the sum of capacities and a cross-term: 
for any two compact sets $A$ and $B$, and for any $r>0$ 
with $A\cup B\subset \B(0,r)$, 
\begin{equation}
\label{decomposition.cap}
\cp(A\cup B) = \cp(A) + \cp(B) - \chi_r(A,B) - \varepsilon_r(A,B),
\end{equation}
with
\begin{equation}
\label{chi.r}
\chi_r(A,B) = 2\pi^2 \, r^2 \cdot \frac 1{|\partial \B(0,r)|} 
\int_{\partial \B(0,r)} (\prstart{H_A<H_B<\infty}{z} + \prstart{H_B<H_A<\infty}{z}) \, dz,
\end{equation}
and
\begin{align}\label{eq:epsilon-r}
\varepsilon_r(A,B) = 2\pi^2 r^2\cdot \frac 1{|\partial \B(0,r)|} 
\int_{\partial \B(0,r)} \prstart{H_A= H_B<\infty}{z}  \, dz,
 \end{align}
where we use the notation $\B(0,r)$ for the
ball of radius $r$ and $\partial \B(0,r)$ for its boundary. 
In particular $ \varepsilon_r(A,B) \le \cp(A\cap B)$. 
The decomposition formula \reff{decomposition.cap} is of a different 
nature to the one presented in~\cite{ASS} for the discrete setting. 
As an illustration, a key technical estimate here concerns
the cross term $\chi_r(A,B)$ where $A$ and $B$ are independent sausages.
In order to bound its first moment, we prove an estimate on the probability of intersection of a Wiener sausage by two other independent Brownian motions.

\begin{proposition}
\label{lem.hit.3}
Let $\beta$, $\gamma$ and $\widetilde \gamma$ be three independent Brownian motions. For any $\alpha>0$ and $c\in (0,1)$, there exist positive constants $C$ and $t_0$, 
such that for all $t>t_0$ and all $z,z'\in \R^4$, with $\sqrt t\cdot (\log t)^{-\alpha}\le \|z\|,\|z'\| \le \sqrt t \cdot(\log t)^\alpha$,  
\begin{equation}\label{hit.zz'}
\bP_{0,z,z'} (W_1[0,t] \cap \gamma[0,\infty)\not=\emptyset ,\  W_1[0,t]
 \cap\widetilde \gamma[0,\infty)\not=\emptyset) \, \le C\,   \frac{(\log \log t)^4}{(\log t)^2}
 \, (1\wedge \frac{t}{\|z'\|^2})\, (1\wedge \frac{t}{\|z\|^2}),
\end{equation}
where $\bP_{0,z,z'}$ means that $\beta$, $\gamma$ and $\til{\gamma}$ start from $0$, $z$ and $z'$ respectively.
\end{proposition}

We note that the problem of obtaining a law of large numbers 
for the capacity of the Wiener sausage has been raised recently 
by  van den Berg, Bolthausen and den Hollander~\cite{BBH16} in connection 
with the torsional rigidity of the complement
of the Wiener sausage on a torus.

The paper is organised as follows. Section~\ref{sec:prelim} 
contains preliminary results: in Section~\ref{sec:not} 
we gather some well-known facts about Brownian motion, and in Section~\ref{subsec.cap}
we prove~\eqref{decomposition.cap} and compare the capacity of a Wiener sausage to 
its volume.  In Section~\ref{Sec.expcap} we prove the asymptotic for the 
expected capacity. In Section~\ref{Sec.updev},
we deduce our large deviation bounds \reff{cor-LD}.
In Section~\ref{subsec.intersection} we provide some 
intersection probabilities of a Wiener sausage 
by another Brownian motion, and 
deduce a second moment bound of the cross-terms $\chi_r$ 
appearing in the decomposition~\eqref{decomposition.cap}.   
Finally, we prove  
Theorem~\ref{theo.cap.wiener} in Section~\ref{sec:finalproof}. 

\section{Preliminaries}\label{sec:prelim}
\subsection{Notation and basic estimates}\label{sec:not}
We denote by $\bP_z$ the law of a Brownian motion starting from $z$, and simply write $\bP$ when $z$ is the origin. 
Likewise $\bP_{z,z'}$ will denote the law of two independent Brownian motions starting respectively from $z$ and $z'$, and similarly for $\bP_{z,z',z''}$. 
For any $x\in \R^4$ and $r>0$, we denote by $\B(x,r)$ the ball of 
radius $r$ centered at $x$.  
We write $|A|$ for the Lebesgue measure of a Borel set $A$. 
We denote by $\|\cdot\|$ the Euclidean norm and by $p_s(x,y)$ the transition kernel of the Brownian motion: 
$$p_s(x,y)= \frac{1}{4\pi^2 s^2}e^{-\frac{\|x-y\|^2}{2s}}=p_s(0,y-x).$$
The Green's function is defined by 
$$G(x,y)=\int_0^\infty p_s(x,y)\ ds := G(y-x).$$
We recall, see Theorem~3.33 in \cite{BM}, that for all $x\neq 0$, 
\begin{equation}
\label{green}
G(x)\ =\  \frac 1{2\pi^2}\cdot \frac 1{\|x\|^2}. 
\end{equation}
We will also write
$$G_t(x):=\int_0^t p_s(0,x)\, dx.$$
Remember now that  
for any $z\in \R^4$, with $\|z\|>r$ (see Corollary 3.19 in \cite{BM}),   
\begin{equation}
\label{hit.ball}
\prstart{H_{\mathcal B(0,r)}<\infty}{z}  = \frac{r^2}{\|z\|^2}. 
\end{equation}
We also need the following well-known estimates. There exist positive constants $c$ and $C$, such that for any $t>0$ and $r>0$, 
\begin{equation}
\label{confinment1}
\pr{\sup_{s\le t} \|\beta_s\| > r }\ \le\ C\cdot \exp(-c\, r^2/t),
\end{equation}
and 
\begin{equation}
\label{confinment2}
\pr{\sup_{s\le t} \|\beta_s\| \le  r}\ \le\ C\cdot \exp(-c\, t/r^2). 
\end{equation}
Finally, we recall the basic result (see Corollary 8.12 and 
Theorem 8.27 in \cite{BM}): 
\begin{lemma}\label{cond.hit}
Let $A$ be a compact set in $\R^4$. Then for any $x\in \R^4\backslash A$, 
$$\prstart{H_A<\infty}{x}  \ \le \ \frac 1{2\pi^2\, d(x,A)^2} \cdot \cp(A),$$
where $d(x,A):=\inf \{\|x-y\|\, :\, y\in A\}$. 
\end{lemma}

\subsection{On capacity}\label{subsec.cap}
We first give a representation formula for the capacity of a set, which has the advantage of not being given as a limit. 
If $A$ is a compact subset of $\R^4$, with $A\subset \B(0,r)$ for some $r>0$,
then
\begin{eqnarray}
\label{cap.formula}
\nonumber \cp(A)& =& \lim_{\|x\|\to \infty} \frac{\prstart{H_A<\infty}{x}}{G(x)} = \lim_{\|x\|\to \infty} \frac{\prstart{H_{\partial \B(0,r)}<\infty}{x}}{G(x)} \cdot  \int_{\partial \B(0,r)} \prstart{H_A<\infty}{z} \, d\rho_x(z)\\
&=& 2\pi^2\, r^2\cdot \,  
\int_{\partial \B(0,r)} \prstart{H_A<\infty}{z} \, d\lambda_r(z),
\end{eqnarray}
where $\rho_x$ is the law of the Brownian motion starting from $x$ at time $H_{\partial \B(0,r)}$, conditioned on this hitting time being finite, and $\lambda_r$ is the uniform measure on $\partial \B(0,r)$. 
The second equality above follows from the Markov property, 
and the last equality expresses the fact that 
the harmonic measure from infinity of a ball, which by Theorem 3.46 in \cite{BM} is also the weak limit of $\rho_x$ as $x$ goes to infinity, 
is the uniform measure on the boundary of the ball.

The decomposition formula~\eqref{decomposition.cap} for the capacity of the union of two sets follows immediately using~\eqref{cap.formula}.

\vspace{0.2cm}
Now we state a lemma which bounds the capacity of the intersection of two Wiener sausages
by the volume of the intersection of larger sausages. 
\begin{lemma}
\label{lem.epsilon}
Let $W$ and $\widetilde W$ be two  independent Wiener sausages. 
Then, almost surely, for all~$t>0$, 
\begin{equation}\label{cap-vol1}
\cp(W_1[0,t]) \ \le \ C_1 \cdot |W_{4/3}[0,t]|,
\end{equation} 
and 
\begin{equation}\label{cap-vol2}
\cp(W_1[0,t] \cap \widetilde W_1[0,t]) \ 
\le\ C_1\cdot | W_4[0,t]\cap \widetilde W_4[0,t]|.
\end{equation}
with $C_1=\cp(\B(0,4))/|\B(0,4/3)|$. 
Moreover, there is a constant $C_2>0$, such that for all $t\ge 2$, 
\begin{equation}\label{cap-vol3}
\E{\cp^2(W_1[0,t]\cap \widetilde W_1[0,t])}\ \le\ C_2\, (\log t)^2.
\end{equation}
\end{lemma}
\begin{proof}[\bf Proof] We start with inequality \reff{cap-vol1}. 
Let $(\mathcal B(x_i,4/3),\ i\le M)$ be a finite covering 
of $W_1[0,t]$ by open balls of radius $4/3$ 
whose centers are all assumed to belong to $\beta[0,t]$, 
the trace of the Brownian motion driving $W_1[0,t]$. 
Then, by removing one by one some balls if necessary, 
one can obtain a sequence of disjoint balls $\B(x_i,4/3, i\le M')$, 
with $M'\le M$, 
such that the enlarged balls $(\mathcal B(x_i,4),\ i\le M')$ 
still cover $W_1[0,t]$. Since the capacity is subadditive, one has on one hand
$$\cp(W_1[0,t])\ \le \ M' \cdot \cp(\B(0,4)),$$
and on the other hand since the balls $\B(x_i,4/3)$ 
are disjoint and are all contained in $W_{4/3}[0,t]$,  
$$
M'| \B(0,4/3)|\ \le\ |W_{4/3}[0,t]|.
$$
Inequality \reff{cap-vol1} follows. Inequality \reff{cap-vol2} is similar:
start with $(\mathcal B(x_i,4/3),\ i\le M)$ a
finite covering of $W_1[0,t]\cap \widetilde W_1[0,t]$ by balls of 
radius one whose centers are all assumed to belong to $\beta[0,t]$.  
Then, by removing one by one some balls if necessary,
one obtain a sequence of disjoint balls $(\mathcal B(x_i,4/3))_{i\le M'}$, such that the enlarged balls $(\mathcal B(x_i,4))_{i\le M'}$ 
cover the set $W_1[0,t]\cap \widetilde W_1[0,t]$, and such that all of them intersect $W_1[0,t]\cap \widetilde W_1[0,t]$. 
But since the centers $(x_i)$ also belong to $\beta[0,t]$, all the balls $\mathcal B(x_i,4/3)$ belong to the enlarged 
intersection $W_4[0,t]\cap \widetilde W_4[0,t]$. 
So as before one has on one hand 
$$
\cp(W_1[0,t]\cap \widetilde W_1[0,t])\ 
\le \ M' \cdot \cp(\mathcal B(0,4)),
$$
and on the other hand 
$$
|W_4[0,t]\cap \widetilde W_4[0,t]|\ \ge \ M'|\mathcal B(0,4/3)|.
$$
We now prove \reff{cap-vol3}. We start with a first moment bound 
(see \cite{Get} for more precise asymptotics):  
\begin{equation}\label{getoor-1}
\E{|W_1[0,t]\cap \widetilde W_1[0,t]|}\le C \log t. 
\end{equation}
This estimate is easily obtained: indeed by definition
\begin{equation}\label{getoor-2}
\E{|W_1[0,t]\cap \widetilde W_1[0,t]|}=\int_{\R^4} 
\bP\left(H_{\B(z,1)}<t\right)^2\, dz, 
\end{equation}
and then \eqref{getoor-1} follows from \eqref{hit.ball} and \reff{confinment1}. 
For the second moment, we write similarly 
\begin{align}\label{eq:initiale}
	\E{|W_1[0,t]\cap \til{W}_1[0,t]|^2} = \int_{\R^4}\int_{\R^4} \pr{H_{\B(z,1)}<t, H_{\B(z',1)}<t}^2\,dz\,dz'.
\end{align}
We now have 
\begin{align*}
	\pr{H_{\B(z,1)}<t, H_{\B(z',1)}<t} = \pr{H_{\B(z,1)}< H_{\B(z',1)}<t} + \pr{H_{\B(z',1)}< H_{\B(z,1)}<t},
\end{align*}
and hence taking the square on both sides gives
\begin{align}\label{eq:square}
	\pr{H_{\B(z,1)}<t, H_{\B(z',1)}<t}^2 \leq  2\pr{H_{\B(z,1)}< H_{\B(z',1)}<t}^2 + 2 \pr{H_{\B(z',1)}< H_{\B(z,1)}<t}^2.
\end{align}
Let $\nu_z$ denote the hitting distribution of the ball $\B(z,1)$ by a Brownian motion starting from~$0$. Then by the strong Markov property we get
\begin{align*}
	\pr{H_{\B(z,1)}< H_{\B(z',1)}<t}\ \leq \ \pr{H_{\B(z,1)}<t}\, \prstart{H_{\B(z',1)}<t}{\nu_z}.
\end{align*}
Substituting this and~\eqref{eq:square} into~\eqref{eq:initiale} gives
\begin{align*}
	\E{|W_1[0,t]\cap \til{W}_1[0,t]|^2}\ \leq\ 4 \int_{\R^4}\int_{\R^4} \pr{H_{\B(z,1)}<t}^2\prstart{H_{\B(z',1)}<t}{\nu_z}^2\,dz\,dz'.
\end{align*}
Using~\eqref{getoor-2} we now obtain for all $z$, 
\begin{align*}
	\int_{\R^4} \prstart{H_{\B(z',1)}<t}{\nu_z}^2\,dz' = \E{|W_1[0,t]\cap \til{W}_1[0,t]|}.
\end{align*}
This together with~\eqref{getoor-1} implies
\begin{align*}
	\E{|W_1[0,t]\cap \til{W}_1[0,t]|^2} \ \leq \ 4\E{|W_1[0,t]\cap \til{W}_1[0,t]|}^2 \ \leq\  4(C \log t)^2,
\end{align*}
and concludes the proof of the lemma. 
\end{proof}


\section{On the Expected Capacity}\label{Sec.expcap}
\subsection{Statement of the result and sketch of proof}
The principal result of this section gives the precise asymptotics 
for the expected capacity.

\begin{proposition}\label{prop-expected} 
In dimension four, and for any radius $r>0$, we have
\begin{align}\label{asymp-expected}
\lim_{t\to\infty}\  \frac{\log t}{t}\, \E{\cc{W_r[0,t]}}\ =\ 4\pi^2.
\end{align}
\end{proposition}

\begin{remark}{\em
The scale invariance of Brownian motion yields in dimension four, for any $r>0$, 
\[
\E{\cc{W_r[0,t]}}\ =\ r^2\, \E{\cc{W_1[0,t/r^2]}}.
\]
Thus, it is enough to prove \reff{asymp-expected} for $r=1$.
}
\end{remark}

The proof is based on an idea of 
Lawler~\cite{Law82} used in the random walk setting. This idea
exploits the fact that the conditional expectation 
of the number of times when 
two random walks meet, conditionally on one of them, is concentrated. 
Before giving the proof, let us explain its rough ideas.

We now give an overview of the proof by introducing the necessary key notation and definitions. We start by discretising the 
Brownian motion~$\beta$ driving $W_1$. 
For fixed $\delta>0$, we record the times
and positions at which $\beta$ leaves successive balls of radius $\delta$.
More precisely, let $\tau_0^\delta=0$,
and by way of induction when $\tau^\delta_i<\infty$, let $Z_i^\delta=\beta(\tau^\delta_i)$ and 
\[
\tau^\delta_{i+1}=\inf\{s>\tau^\delta_i\, :\, \beta_s \notin \B(Z_i^\delta,\delta) \}.
\]
It follows from \eqref{confinment2} that for any $i$, the stopping time $\tau^\delta_{i}$ is almost surely finite.
Then, we define the Wiener sausage associated to the discrete positions by   
\[
W_r^\delta[0,t]:=\bigcup_{i: \, \tau^\delta_i\leq t} \B(Z^\delta_i,r), \quad \text{for all } r\ge 0.
\]
The continuity of the Brownian path implies that  
$\|Z^\delta_i-Z^\delta_{i-1}\|=\delta$, almost surely for all $i\ge 1$. Therefore, one has for all $\delta>0$, 
\begin{eqnarray}\label{cap.inclusion}
W_1^\delta[0,t]\, \subseteq \, W_1[0,t]\,  \subseteq\,  W_{1+\delta}^\delta[0,t]. 
\end{eqnarray}
Moreover, the following scaling relation holds in law,
\begin{eqnarray*}\label{scaling-key}
\cc{W_{1+\delta}^\delta[0,t]}\ \stackrel{\text{(law)}}{=}\ 
(1+\delta)^{d-2} \, 
\cc{W_{1}^{\frac{\delta}{1+\delta}}\left[0,\frac{t}{(1+\delta)^2}\right]}.
\end{eqnarray*}
Thus, using that the capacity is monotone 
for inclusion, it is enough to obtain asymptotics 
for the expected capacity of $W_1^\delta[0,t]$, 
and then let $\delta$ go to zero.

The next step is to cast the expected capacity of $W_1^\delta[0,t]$ into a probability of non-intersection of this discretised  Wiener sausage by another Brownian motion $\til \beta$, starting from infinity. More precisely we will show below that 
\begin{align}\label{def-exp}
\E{\cc{W_1^\delta[0,t]}}\, =\, \lim_{\|z\|\to\infty}\, \frac{1}{G(0,z)}\cdot
\prstart{W_1^\delta[0,t]\cap\til \beta[0,\infty)\not= \emptyset}{0,z},
\end{align}
which should not come as a surprise, 
since this formula holds for deterministic sets \eqref{cap.def} (but one still need to justify the interchange of limit and expectation). 
We next introduce the following stopping time
\begin{equation}\label{def.tau}
\tau\ =\ \inf\{s\ge 0\ : \ \til \beta_s \in W_1^\delta[0,t]\},
\end{equation}
and note that the probability on the right-hand side 
of \eqref{def-exp} is just the probability of $\tau$ being finite.

Then we introduce a counting measure of the pairs of times at which 
the two trajectories come within distance $1$ 
\begin{align}\label{def-Rdelta}
R^\delta[0,t]= \sum_{i\ge 0}
\big(\tau^\delta_{i+1}\wedge t-\tau^\delta_{i}\wedge t\big)
\int_0^\infty \1(\|\til \beta_s-Z_{i}^\delta\|\le 1)\, ds.
\end{align}
Observe that $\tau$ is finite, if and only if, $R^\delta[0,t]$ is nonzero.
Therefore the following equality holds
\begin{align}\label{exp-5}
\prstart{\tau<\infty}{0,z}\ =\ \frac{\estart{R^\delta[0,t]}{0,z}}
{\estart{R^\delta[0,t] \mid  \tau<\infty}{0,z}}.
\end{align}
The estimate of the numerator in \eqref{exp-5} is established by 
comparing $R^\delta[0,t]$ to a continuous counterpart $R[0,t]$, whose expectation can be computed explicitly and which is defined via    
\begin{align}\label{def-R}
R[0,t]\, =\, \int_0^\infty ds\int_0^t
\, \1(\|\til \beta_s-\beta_u\| \le 1)\, du .
\end{align}
More precisely we prove in Lemma \ref{lem-ER} 
below (see Subsection \ref{Subsec:proof}) that for all $t>0$, 
\begin{eqnarray}\label{exp-R}
\lim_{\|z\|\to\infty}\, \frac{\estart{R[0,t]}{0,z} }{G(0,z)} 
\ =\ \frac{\pi^2}{2} t,
\end{eqnarray}
The same limit holds for $R^\delta[0,t]$, 
up to some additional $\mathcal O(1)$ term. 
The estimate of the denominator in \eqref{exp-5} is more intricate. 
Consider the random time
\begin{align}\label{time-sigma}
\sigma\, =\, \inf\left\{ i\geq 0\, : \, \|\til{\beta}(\tau)-Z_i^\delta\| \leq
 1\right\}.
\end{align}
A key observation is that $\sigma$ is not a stopping time (with respect to any natural filtration), since~$\tau$  
depends on the whole Wiener sausage $W_1^\delta[0,t]$. In other words  conditionally on $\tau$ and $\sigma$, one cannot consider the two parts of the 
trajectories of $\til \beta$ and $W_1^\delta$ after the times $\tau$ and $\sigma$ respectively, as being 
independent
\footnote{a mistake 
that Erd\"os and Taylor implicitly made in their pioneering work \cite{ErdosTaylor}, and that Lawler corrected about twenty years later \cite{Law82}.}.

To overcome this difficulty, the main idea (following Lawler) is to use that $\mathbb E[R^\delta[0,t]\mid (\beta_s)_{s\le t}]$ is  concentrated around its mean value, which is of order $\log t$.  
As a consequence, even if the trajectory of $\beta$ after time $\tau_\sigma^\delta$ 
is not independent of $\widetilde \beta[\tau, \infty)$, we still have that 
$\mathbb E_{0,z}[R^\delta\mid \beta]$ estimated in the time
period $[\tau_\sigma^\delta,t]$,
is close to its mean value for typical values of~$\sigma$. 
Another difficulty then is to control the probability for $\sigma$ to be typical with this respect (what Lawler calls a {\it good} $\sigma$), and the solution is inspired by another nice argument of Lawler. But we refer to the proof below for more details.

However, there are some small additional issues here. 
Unlike in the discrete case, $Z_\sigma$ and $\til \beta(\tau)$ 
are not equal.
In particular $\mathbb E_{0,z}[R^\delta[0,t]\mid \beta,\, (\til \beta_s)_{s\le \tau}]$ 
is not distributed as $\mathbb E_{0,0}[R^\delta[0,t-\tau_\sigma^\delta]\mid \beta]$, but as $\mathbb E_{0,x}[R^\delta[0,t-\tau_\sigma^\delta]\mid \beta]$, with $x= \til \beta(\tau)-Z_\sigma$, which is nonzero. However, since we still have that $\|x\|=1$, one can compare this expectation, with the one when $x=0$ and show that their difference is negligible.  

Another point is that the argument described above requires $\log t$ and $\log (t-\tau_\sigma^\delta)$ to be equivalent, at least when we look for an upper bound of the probability that $\tau$ is finite. 
A simple way to overcome this difficulty is to work with a longer period, 
and use instead of \eqref{exp-5} the inequality
$$
\prstart{\tau<\infty}{0,z}\ \le \ 
\frac{\estart{R^\delta[0,t(1+\varepsilon)]}{0,z}}
{\estart{R^\delta[0,t(1+\varepsilon)] \mid  \tau<\infty}{0,z}},
$$
which holds for any fixed positive $\epsilon$, 
and then let $\epsilon$ go to zero. 
This concludes our overview of the proof,
which now starts.

\subsection{Proof of Proposition~\ref{prop-expected}}
\label{Subsec:proof}

The first thing is to prove \eqref{def-exp}. 
For any real $\rho>0$, with $d\lambda_\rho$
denoting the uniform probability measure on the boundary of $\B(0,\rho)$,
we have shown in~\eqref{cap.formula} that
\begin{eqnarray*}
\cc{W_1[0,t]\cap \B(0,\rho)}=\frac{1}{G(0,2\rho)}\int_{\partial \B(0,2\rho)}
\prstart{W_1[0,t]\cap \B(0,\rho)\cap\til \beta[0,\infty)
\not= \emptyset\ \big|\ W_1[0,t]}{0,z}\, d\lambda_{2\rho} (z).
\end{eqnarray*}
Taking expectation on both sides we obtain
\begin{eqnarray*}
\E{\cc{W_1[0,t]\cap \B(0,\rho)}}=\frac{1}{G(0,2\rho)}\int
\prstart{W_1[0,t]\cap \B(0,\rho)\cap\til \beta[0,\infty)
\not= \emptyset}{0,z} \, d\lambda_{2\rho} (z).
\end{eqnarray*}
By rotational invariance of $\beta$ and $\til{\beta}$, we get that the probability appearing in the integral above is the same for all $z\in \partial \B(0,2\rho)$. Writing $2\rho=(2\rho,0,\ldots,0)$ we get
\begin{align*}
	&\E{\cc{W_1[0,t]\cap \B(0,\rho)}}=\frac{1}{G(0,2\rho)}
\prstart{W_1[0,t]\cap \B(0,\rho)\cap\til \beta[0,\infty) 
\not= \emptyset}{0,2\rho} \\&=
\frac{1}{G(0,2\rho)}
\prstart{W_1[0,t]\cap\til \beta[0,\infty) 
\not= \emptyset}{0,2\rho} + \mathcal O\Big(\frac{\pr{W_1[0,t]\cap \B^c(0,\rho)
\not= \emptyset}}{G(0,2\rho)}\Big).
\end{align*} 
Using that the $\mathcal O$ term appearing above tends to $0$ as $\rho\to\infty$ and invoking monotone convergence proves~\eqref{def-exp}.

Now in view of \eqref{exp-5} we need to estimate the expectation of $R^\delta[0,t]$ conditionally on $\tau$ being finite. 
To this end, we take the expectation conditionally on $(\beta_s)_{s\ge 0}$ and introduce the following random variables: 
$$
D_x[0,t]=\estart{R[0,t]\mid \beta}{0,x}  \quad \text{and}\quad D_x^\delta[0,t]=\estart{R^\delta[0,t]\mid \beta}{0,x}.
$$
Note that if we set
$$
G^*(x,y)=\int_{\B(y,1)}\!\! G(x,z)\, dz  =G^*(0,y-x),
$$
then,
\begin{eqnarray}\label{Dx}
D_x[0,t] =\int_{\R^4}\, dy\int_0^t\,  G(x,y)
\1( \|y-\beta_s\|\le 1)\, ds =\int_0^t G^*(x,\beta_s) \, ds, 
\end{eqnarray}
and 
\begin{eqnarray}
\label{Dxdelta}
D_x^\delta[0,t]\, =\, \sum_{i\ge 0} (\tau^\delta_{i+1}\wedge t-\tau^\delta_i\wedge t)\,  
G^*(x,Z^\delta_i).
\end{eqnarray}
To simplify notation, we write 
$$d(t) = \mathbb E[D_0[0,t]],$$ 
since this quantity will play an important role in the rest of the proof. 
Now before we continue with the proof we gather here some technical results that will be needed.

The first result contains \eqref{exp-R} and also provides an estimate of the difference between $R^\delta[0,t]$ and~$R[0,t]$. 
\begin{lemma}\label{lem-ER} 
There exists a constant $C>0$, such that for all $\delta\le 1$, $z\neq 0$ and $t\ge 1$, 
\[
\left|\frac{\estart{R^\delta[0,t]}{0,z}}{G(0,z)}-\frac{\estart{R[0,t]}{0,z}}{G(0,z)}\right|\ \le\ C\, \left( 1+ \frac{t}{\norm{z}} + \norm{z}e^{-\frac{\norm{z}^2}{8t}}\right).
\]
Moreover, for all $t>0$, 
\begin{align*}
\lim_{\|z\|\to\infty}\ \frac{\estart{R[0,t]}{0,z} }{G(0,z)} \ =\ \frac{\pi^2}{2}\,  t.
\end{align*}
\end{lemma}

The second result deals with the first and second moments of $D_0[0,t]$. 
\begin{lemma}\label{lem-D0}
One has
$$\lim_{t\to \infty} \ \frac 1{\log t} \, \mathbb E[D_0[0,t]]\ = \ \frac 18,$$
and there exists a constant $C>0$, such that for all $t\ge 2$, 
\begin{align}\label{2moment-D}
\E{D_0[0,t]^2}\ \le\  C \, (\log t)^2.
\end{align}
\end{lemma}

The third result shows that $D_x^\delta[0,t]$ is uniformly close to $D_0[0,t]$, when $\|x\|$ is smaller than one: 

\begin{lemma}\label{lem-D} Let   
\begin{align}\label{def-zeta}
\zeta=\int_0^\infty \frac{1}{\norm{\beta_s}^3\vee 1} \,ds.
\end{align}
Then the following assertions hold. 
\begin{enumerate}
\item[(i)]
There exists a constant $\lambda>0$, such that 
\[
\E{\exp(\lambda\, \zeta)} \ <\ \infty. 
\]
\item[(ii)] There exists a constant $C>0$, so that for all $\delta\le 1$ and $t>0$, almost surely, 
\[
\sup_{\norm{x}\leq 1} \left| D^\delta_x[0,t]-D_0[0,t]\right|\ \leq \ C\, \zeta.
\]
\end{enumerate}
\end{lemma}
The next result gives some large deviation bounds for $D_0[0,t]$, and shows that it is concentrated. 
\begin{lemma}\label{LD.D0}
For any $\epsilon>0$,  there exists $c=c(\epsilon)>0$, such that for $t$ large enough, 
\[
\pr{|D_0[0,t] - d(t)|>\epsilon\, d(t) } \ \le\  \exp\big(-c \, (\log t)^{1/3} \big),
\]
where we recall that $d(t)=\mathbb E[D_0[0,t]]$. 
\end{lemma}

Finally the last preliminary result we should need is the following elementary fact:  
\begin{lemma}\label{lem:infint}
There exists a constant $C>0$, so that for all $k\in \N$ and $z\in \R^4$,
\[
\pr{\inf_{k\leq s\le k+1} \|\til{\beta}_s-z\|\leq 1} \ \leq \ C \, \int_{k}^{k+2}
\pr{\|\til{\beta}_u-z\|\leq 2}\, du.
\]
\end{lemma}

The proofs of these five lemmas are postponed to Sections~\ref{sec:lemmas} and~\ref{subsec:LD.D0}, and assuming them one can now finish the proof of Proposition \ref{prop-expected}.

Denote by $(\F^\beta_s)_{s\ge 0}$ and $(\F^{\til \beta}_s)_{s\ge 0}$ the natural filtrations of $\beta$ and $\til \beta$ respectively. Recall the definition \eqref{def.tau} of $\tau$, and then define the sigma-field $\mathcal G_\tau:=\F^{\til \beta}_\tau \vee (\F^\beta_s)_{s\ge 0}$.  
Next, recall the definition \eqref{def-Rdelta} of $R^\delta[0,t]$, and observe that on the event $\{\tau<\infty\}$, we have
\begin{eqnarray}\label{Rdelta.Gtau}
\estart{R^\delta[0,t]\mid \mathcal G_\tau}{0,z}\ =\ \sum_{j\ge \sigma}
\big(\tau^\delta_{j+1}\wedge t-\tau^\delta_j\wedge t\big) \, G^*(\til \beta(\tau)- Z_\sigma^\delta,Z^\delta_j-Z_\sigma^\delta),  
\end{eqnarray}
since indices $j$ smaller than $\sigma$ contribute zero in the sum by the definition of $\sigma$.

Now recall that our goal is to estimate the probability of $\tau$ being finite. We divide the proof in two parts, one 
for the lower bound and another one for the upper bound, and give two definitions of {\it good} $\sigma$ accordingly.

\underline{Proof of the lower bound.} We fix some $\epsilon>0$, and define an integer $i$ to be {\it good} if
\begin{eqnarray*}
\sup_{\norm{x}\leq 1}\,   \sum_{j\ge i}\,  
(\tau^\delta_{j+1}\wedge (\tau_i^\delta +t)-\tau^\delta_j\wedge (\tau_i^\delta +t)) G^*(x,Z^\delta_j - Z^\delta_i) 
\ \leq\ (1+\epsilon)\, d(t),
\end{eqnarray*}
and otherwise we say that $i$ is {\it bad}. Observe that the event $\{i \text{ good}\}$ is $\sigma((\beta_s-\beta(\tau^\delta_i))_{s\ge \tau_i^\delta})$-measurable, and in particular is independent of $(Z_k^\delta)_{k\le i}$. Note also that it depends 
in fact on~$t$ and $\varepsilon$, but since they are kept fixed in the rest of the proof this should not cause any confusion. 
Moreover, by the strong Markov property applied to $\tau_i^\delta$, one has (recall \eqref{Dxdelta})   
\begin{eqnarray}\label{ibad1}
\nonumber \prstart{i \text{ bad}}{0,z} & =&   \pr{\sup_{\norm{x}\leq 1}  D_x^\delta[0,t] > (1+\epsilon) \, d(t)} \\ 
&\le & C\, \exp(-c\, (\log t)^{1/3}),
\end{eqnarray}
for some positive constants $c$ and $C$, 
where the last inequality follows from Lemmas~\ref{lem-D0}, \ref{lem-D} and~\ref{LD.D0}.  
When $\tau$ is finite, then the event $\{\sigma \text{ good}\}$ can be written as 
$$\{\sigma \text{ good}\}\ =\ \bigcup_{i\ge 0} \ \{\sigma = i\} \cap \{i\text{ good}\}.$$
We also denote the complementary event as $\{\sigma \text{ bad}\}$. A subtle and difficult point here is that one cannot proceed as in \eqref{ibad1}. 
Indeed, the problem is that the events $\{\sigma=i\}$ and $\{i \text{ bad}\}$ are not independent. So the idea of Lawler, see \cite[page 101]{Law91}, in the random walk setting, 
was  to decompose the event $\{\sigma \text{ bad} \}$ into all the possible values for $\sigma$ and $\tau$ (in our case we will discretise $\tau$ and consider 
all the possible values of its integer part), and loosely replace the event $\{\sigma = i,\, \tau = k\}$ 
by the event that the two walks are at the same position 
at times $i$ and $k$ respectively.
The interest of doing so is that now the latter event is independent of the event $\{ i \text{ bad}\}$ and probabilities factorise. 
The remaining part is a double sum which is equal (in the discrete case) to the expected number of pairs of times the two walks coincide. 
In our case, Lemma \ref{lem:infint} will show that the remaining double sum can be compared with the expectation of $R^\delta[0,t]$. 
As it turns out, we will see that this argument is not too loose, since the probability 
of $i$ being bad is sufficiently small: it decays as a stretched exponential in $\log t$, as \eqref{ibad1} tells us. 
But we will come back to this in more detail a bit later, see \eqref{bad-1} and \eqref{bad-9} below.

For the moment, just observe that the event $\{\sigma\text{ good}\}$ is $\mathcal G_\tau$-measurable. 
As a consequence, one has 
\begin{eqnarray}\label{Rdelta.sigma.good}
\nonumber \mathbb E_{0,z} [R^\delta[0,t] \1(\tau<\infty,\, \sigma \text{ good})] &=& \mathbb E_{0,z} \left[\mathbb E_{0,z} [R^\delta[0,t]\mid \mathcal G_\tau] \cdot \1(\tau<\infty,\, \sigma \text{ good})\right]\\
&\le & (1+\epsilon)\, d(t) \, \mathbb P_{0,z}(\tau<\infty,\, \sigma \text{ good}),
\end{eqnarray}
where the last inequality follows from \eqref{Rdelta.Gtau}, the definition of $\sigma$ good and the easy fact that
\[
\tau_{j+1}^\delta\wedge t-\tau_j^\delta\wedge t \le
\tau_{j+1}^\delta\wedge(\tau_i^\delta+ t)-\tau_j^\delta\wedge(\tau_i^\delta +t).
\]
Therefore, we can write  
\begin{eqnarray}\label{lower-1}
\nonumber \prstart{\tau<\infty}{0,z} & \ge & 
\prstart{\tau<\infty, \sigma \text{ good}}{0,z}\\
\nonumber & = &  \frac{\estart{R^\delta[0,t] \1(\tau<\infty, \, \sigma \text{ good})}{0,z}}
{\escond{R^\delta[0,t]}{\tau<\infty, \sigma\text{ good}}{0,z}}\\
& \ge  &   \frac1 {(1+\epsilon)d(t)}\cdot \estart{R^\delta[0,t] \1(\tau<\infty,\, \sigma \text{ good})}{0,z}. 
\end{eqnarray}
The last term above is estimated through
\begin{align}\label{lower-3}
\estart{R^\delta[0,t]\1(\tau<\infty,\, \sigma \text{ good})}{0,z}=
\estart{R^\delta[0,t]}{0,z}-
\estart{R^\delta[0,t]\1(\tau<\infty, \sigma \text{ bad})}{0,z}.
\end{align}
Now the idea for estimating the expectation of $R^\delta[0,t]$ on the event $\sigma$ bad, is to use the strategy of Lawler,  described earlier. Using furthermore \eqref{Rdelta.Gtau}, and letting $y=\til \beta(\tau)- Z_\sigma^\delta$, we can write   
\begin{eqnarray}\label{bad-1}
\nonumber &&\estart{R^\delta[0,t]\1(\tau<\infty, \sigma \text{ bad })}{0,z}   \ =\  \mathbb E_{0,z}\left[\mathbb 
E_{0,z}[R^\delta[0,t]\mid \mathcal G_\tau]\, \1(\tau<\infty,\, \sigma \text{ bad}) \right] \\ 
\nonumber & =&  \mathbb E_{0,z}\left[ \left( \sum_{j\ge \sigma} (\tau_{j+1}^\delta\wedge t-\tau_j^\delta\wedge t)G^*(y,Z_j^\delta-Z_\sigma^\delta)\right)\, \1(\tau<\infty,\,  \sigma \text{ bad}) \right] \\
\nonumber &= & \sum_{k=0}^{\infty}\sum_{i=0}^{\infty}\, \mathbb E_{0,z}\left[ \left( \sum_{j\ge i}(\tau_{j+1}^\delta\wedge t-\tau_j^\delta\wedge t)G^*(y,Z_j^\delta-Z_i^\delta)\right)\, \1([\tau]=k,\, \sigma=i, \, i \text{ bad}) \right] \\
\nonumber &\le &  \sum_{k=0}^{\infty}\sum_{i=0}^{\infty}\,  \mathbb E_{0,z}\left[ \left( \sum_{j\ge i} (\tau_{j+1}^\delta\wedge t-\tau_j^\delta\wedge t)G^*(y,Z_j^\delta-Z_i^\delta) \right)\, \1\left(\inf_{k\leq s\leq k+1}\|\til{\beta}_s - Z^\delta_{i}\|
 \leq 1,\, \tau^\delta_i\le t,\, i \text{ bad}\right) \right] \\
 &\leq &    \mathbb E\left[ \left(\sup_{\|x\|\le 1} D_x^\delta[0,t]\right) \, \1(0 \text{ bad})\right]\, \sum_{k=0}^{\infty}\sum_{i=0}^{\infty}\, \mathbb P_{0,z}\left(\inf_{k\leq s\leq k+1}\|\til{\beta}_s - Z^\delta_{i}\| \leq 1,\, \tau^\delta_i\le t\right), 
\end{eqnarray}
using the strong Markov property for $\beta$ at time $\tau_i^\delta$ for the last inequality. 
Using Cauchy-Schwarz, Lemmas~\ref{lem-D0} and~\ref{lem-D} and~\eqref{ibad1} we upper bound the expectation appearing on the last line above to get
\begin{align}\label{bad-5}
 \mathbb E\left[ \left(\sup_{\|x\|\le 1} D_x^\delta[0,t]\right) \, \1(0 \text{ bad})\right]\ \le \ C(\log t)^2\, \exp(-c\, (\log t)^{1/3}),
\end{align}
for some positive constants $c$ and $C$. 
The last double sum in \eqref{bad-1} is dealt with using Lemma \ref{lem:infint}. But before we proceed with it, 
let us define 
$$R_2^\delta[0,t] \ :=\ \sum_{i\ge 0}\,  (\tau_{i+1}^\delta\wedge t-\tau_i^\delta\wedge t) \int_0^\infty \1(\|\til \beta_s-Z_i^\delta\| \le 2)\, ds.
$$
Note that $\mathbb E[\tau_{i+1}^\delta\wedge t - \tau_i^\delta\wedge t] \leq  \delta^2/4$, for all $i\ge 0$. Thus 
$$\mathbb E_{0,z}[R_2^\delta[0,t]]\ \leq \ \frac{\delta^2}{4}\, \sum_{i\ge 0} \int_0^\infty \mathbb P(\|\til \beta_s-Z_i^\delta\| \le 2,\, \tau_i^\delta\le t)\, ds.$$
Together with Lemma  \ref{lem:infint}, we obtain   
\begin{eqnarray*}
\sum_{k=0}^{\infty}\sum_{i=0}^{\infty}\, \mathbb P_{0,z}\left(\inf_{k\leq s\leq k+1}\|\til{\beta}_s - Z^\delta_{i}\| \leq 1,\, \tau^\delta_i\le t\right) \ \le\  \frac{C}{\delta^2}\ \mathbb E_{0,z}[ R^\delta_2[0,t]]. 
\end{eqnarray*}
Moreover, by Brownian scaling, one can see that $R^\delta_2[0,t]$ is equal in law to  
$16R^{\delta/2}[0,t/4]$. Therefore, it follows from Lemma~\ref{lem-ER} 
that there exists a constant $C>0$, such that for all $t\ge 1$ and $\delta \le 1$,  
$$\limsup_{\|z\|\to \infty} \frac{\mathbb E_{0,z}[ R^\delta_2[0,t]]}{G(0,z)} \ \le\  C\, t.$$

Combining this with Lemma \ref{lem-ER}, \eqref{lower-1}, \eqref{lower-3}, \eqref{bad-1} and \eqref{bad-5} gives \begin{eqnarray*}
\liminf_{\|z\|\to \infty}\  \frac{\mathbb P_{0,z}(\tau<\infty)}{G(0,z)} \ \ge \ \frac {\pi^2}{2(1+\epsilon)}\cdot \frac{t}{d(t)}\cdot \left(1-\frac{C}{\delta^2}\, \exp(-c\, (\log t)^{1/3})\right),
\end{eqnarray*}
for all $t$ large enough. 
Using in addition Lemma \ref{lem-D0}, and since the above estimate holds for all $\epsilon>0$, we get (recall \eqref{cap.inclusion} and \eqref{def-exp})   
\begin{eqnarray*}\label{lowerbound}
\liminf_{t\to \infty} \ \frac{\log t}{t} \cdot \mathbb E[ \cp(W_1[0,t])]\   \ge \ 4\pi^2. 
\end{eqnarray*}

\vspace{0.2cm}
\underline{Proof of the upper bound.}

We again fix $\epsilon\in (0,1)$ and we 
define an integer  $i$ to be {\it good} if
\begin{align*}
\inf_{\norm{x}\leq 1} \  \sum_{j\ge i}\, (\tau^\delta_{j+1}\wedge (\tau^\delta_i+\epsilon t)-\tau^\delta_j\wedge (\tau^\delta_i+\epsilon t)) G^*(x,Z^\delta_j - Z^\delta_i) \ 
\geq \ (1-\epsilon)\, d(t),
\end{align*}
and otherwise we say that $i$ is {\it bad}. The probability the latter happens satisfies
\begin{align}\label{bad-4}
\prstart{i \text{ bad}}{0,z}=\pr{\inf_{\norm{x}\leq 1} 
 D_x^\delta[0,\epsilon t] < (1-\epsilon) \, d(t) } \ \le\ C\, \exp(-c\, (\log t)^{1/3}),
\end{align}
where again the last inequality follows from Lemmas~\ref{lem-D0}, \ref{lem-D} and~\ref{LD.D0}.

We write next 
\[
\prstart{\tau<\infty}{0,z} = \prstart{\tau<\infty,\, \sigma \text{ good}}{0,z} + \prstart{\tau<\infty, \, \sigma \text{ bad}}{0,z}.
\]
Let us treat first the probability with the event $\sigma$ good. We have 
\begin{align*}
\prstart{\tau<\infty, \,\sigma \text{ good}}{0,z} \leq
\frac{\estart{R^\delta[0,t(1+\epsilon)]}{0,z}}{\escond{R^\delta[0,t(1+\epsilon)}{\tau<\infty, \sigma \text{ good}}{0,z}}.
\end{align*}
Using the same argument as for the lower bound, the new definition of good $\sigma$ and the fact that on the event $\{\tau_i^\delta\leq t\}$ we have 
\[
\tau_{j+1}^\delta \wedge (t+\epsilon t) - \tau_{j}^\delta\wedge (t+\epsilon t) \geq \tau_{j+1}^\delta \wedge (\tau_i^\delta+\epsilon t) - \tau_{j}^\delta\wedge (\tau_i^\delta+\epsilon t),
\]
we see that 
\begin{align*}
\escond{R^\delta[0,t(1+\epsilon)]}{\tau<\infty, \sigma \text{ good}}{0,z} 
\geq (1-\epsilon) \, d(t).
\end{align*}
Together with Lemma~\ref{lem-ER} this provides the upper bound for the term 
\[
\limsup_{\norm{z}\to \infty} \frac{\prstart{\tau<\infty, \sigma \text{ good}}{0,z}}{G(0,z)}.
\]

Let us treat now the event $\sigma$  bad. Using the same argument as for the lower bound and in particular \eqref{bad-4}, we obtain 
\begin{align}\label{bad-9}
\begin{split}
\prstart{\tau<\infty, \sigma \text{ bad}}{0,z}  \  & 
=\ \sum_{k=0}^{\infty}\sum_{i=0}^{\infty}\,  \prstart{\sigma=i, \, [\tau]=k,\,  i \text{ bad}}{0,z}   \\
& \leq \ \sum_{k=0}^{\infty}\sum_{i=0}^{\infty}\,  
\prstart{\inf_{k\leq s\leq k+1}\|\til{\beta}_s - Z^\delta_i\|  \leq 1,\, \tau^\delta_i\le t,\, i \text{ bad}}{0,z}\\
& \leq\  \sum_{k=0}^{\infty}\sum_{i=0}^{\infty}\, 
\prstart{\inf_{k\leq s\leq k+1}\norm{\til{\beta}_s - Z^\delta_{i}}  \leq 1, \tau^\delta_i\le t}{0,z}\pr{ i \text{ bad}}\\
& \le\ \frac{C}{\delta^2}\,  \estart{R^\delta_2[0,t]}{0,z} \, \exp(-c\, (\log t)^{1/3}),
\end{split}
\end{align}
for some positive constants $c$ and $C$. We conclude similarly as for the lower bound that 
\begin{eqnarray*}
	\limsup_{t\to\infty} \ \frac{\log t}{t}\cdot \E{\cc{W_1[0,t]}} \, \leq\,   4\pi^2,
\end{eqnarray*}
and this completes the proof of Proposition \ref{prop-expected}.
\hfill $\square$


\subsection{Proofs of Lemma~\ref{lem-ER}, \ref{lem-D0}, \ref{lem-D} and~\ref{lem:infint}}\label{sec:lemmas}
Before we start with the proofs, it will be convenient to introduce some new notation. 
For $A\subset \R^4$ measurable, we denote by $\ell(A)$ the total time spent in the set $A$ by the Brownian motion $\beta$: 
$$\ell(A):=\int_0^\infty \1(\beta_s\in A)\, ds.$$
We also define the sets $A_0=\B(0,1)$, and 
$$A_i\ :=\ \B(0,2^i)\setminus\B(0,2^{i-1}),\quad \text{for }i\ge 1.$$
Note that for any $A$ and $k\ge 1$, one has using the Markov property  
\begin{eqnarray}
\label{localtime} 
\mathbb E[\ell(A)^k]  =   k!\, \mathbb E\left[ \int_{s_1\le \dots\le s_k } \1(\beta_{s_1}\in A,\dots,\beta_{s_k} \in A)\, ds_1\dots ds_k\right] \le k! \, \left(\sup_{x\in A\cup \{0\}}\, \mathbb E_x [\ell (A)]\right)^k.
\end{eqnarray}

\begin{proof}[\bf Proof of Lemma~\ref{lem-D}] 
Let us start with part (i). 
Observe first that 
\begin{align*}
\zeta\ \le\  \sum_{i=0}^\infty\,  \frac{\ell(A_i)}{2^{3(i-1)}}.
\end{align*}
Using Jensen's inequality and that $\ell(A_i)$ has the same distribution as  $2^{2(i-1)} \ell(A_1)$ for all $i\geq 1$, we obtain
\begin{align*}
\begin{split}
\mathbb E[\zeta^k]\ \le &\ 4^k\,  \E{\left( \sum_{i=0}^\infty \frac{1}{2^{i+1}} \frac{\ell(A_i)}{2^{2(i-1)}}\right)^k}\\
\le & \ 4^k \, \sum_{i=0}^\infty \frac{1}{2^{i+1}} \, \E{\left( \frac{\ell(A_i)}{2^{2(i-1)}}\right)^k}\\
\le  & \ 16^k\,  \E{\ell(A_0)^k} + 4^k\, \E{\ell(A_1)^k}\\
\le & \ C^k\,  k!,
\end{split}
\end{align*}
for some constant $C>0$, where we used \eqref{localtime} at the last line. The first part of the lemma follows.

Now we prove part (ii). 
	To simplify notation, write $G^*(z) = G^*(0,z)$, and recall that by definition 
\[
G^*(z) =\int_{\B(z,1)} G(w)\, dw.  
\]	
Then recall that the Green's function $G$ is harmonic on $\R^4\smallsetminus \{0\}$, so it satisfies the mean-value property on this domain. This implies that if $\|z\|>1$, then $G^*(z)=|\B(0,1)|\cdot G(0,z)$. Recall furthermore that 
$|\B(0,1)|=\pi^2/2$, so that $G^*(z) = (\pi^2/2) \cdot G(z)=1/(4\|z\|^2)$, when $\|z\|>1$.  	
Now suppose that $\norm{u}> 2$ and $\|x\|\le 1$. Then 
\[
|G^*(u+x) - G^*(u)| =\frac 14\,  \left|\frac{1}{\norm{u+x}^2} - \frac{1}{\norm{u}^2}\right| \leq
 \frac{1+2\norm{u}}{\norm{u+x}^2 \norm{u}^2} \le \frac{C}{\norm{u}^3}.
\]
Since in addition $G^*$ is bounded on $\B(0,3)$, we deduce that there exists  $C>0$, so that for all $u\in \R^4$,
\begin{align}\label{more-7}
\sup_{\norm{x}\leq 1} \left| G^*(u+x) - G^*(u) \right|\leq  \frac{C}{\norm{u}^3 \vee 1}.
\end{align}
Then it follows from the formulas \eqref{Dx} and \eqref{Dxdelta} for $D_x[0,t]$ and $D_x^\delta[0,t]$ respectively, that 
\begin{align*}
\sup_{\norm{x}\leq 1}\, |D_x[0,t] - D_0[0,t]| \ \le\ C \, \int_0^t \frac{1}{\norm{\beta_u}^3\vee 1}\, du, 
\end{align*}
and 
\begin{align*}
	\sup_{\|x\|\le 1} \, | D_x^\delta[0,t]-D^\delta_0[0,t] | \ \leq\ C\,  \sum_{i\ge 0} \,  
	\int_{\tau_i^\delta}^{\tau_{i+1}^\delta}\frac{1}{\|\beta(\tau_i^\delta)\|^3\vee 1} \, ds.
\end{align*}
Moreover, by definition of the times $\tau_i^\delta$ one has $\|\beta_s-\beta(\tau_i^\delta)\| \leq \delta$, for all $s\in [\tau_i^\delta, \tau_{i+1}^\delta]$. 
Therefore for all such $s$, by the triangle inequality 
$\|\beta(\tau_i^\delta)\| \geq  \|\beta_s\|-\delta\ge \|\beta_s\|/2$, as long as $\|\beta_s\|\ge 2\delta$. 
Since $\delta\le 1$, it follows that 
\begin{align*}
	\sup_{\|x\|\le 1}\, | D_x^\delta[0,t]-D_0^\delta[0,t] |  \ \le\  C\, \left(\int_0^\infty \1(\|\beta_s\|\le 2\delta)\, ds +\int_0^\infty \frac{1}{\|\beta_s\|^3\vee 1}\, ds\right)\ \le\ 9 C \zeta.
\end{align*}
It remains to compare $D_0^\delta[0,t]$ with $D_0[0,t]$. But since it makes no difference in the proof, and since in addition we will need it in the proof of Lemma \ref{lem-ER}, we compare in fact $D_z^\delta[0,t]$ with $D_z[0,t]$, for general $z\in \R^4$. We now have
\begin{align*}
| D_z^\delta[0,t] - D_z[0,t]|\  \leq\  \sum_{i\ge 0}\, \,  \int_{\tau^\delta_i\wedge t}^{\tau^\delta_{i+1}\wedge t} |G^*(z+\beta_s) - G^*(z+\beta(\tau^\delta_i))| \,ds.
\end{align*}
Then using again that $\|\beta_s - \beta(\tau^\delta_i)\|\leq \delta$, for all~$s \in [\tau^\delta_i,\tau^\delta_{i+1}]$,  
and \eqref{more-7} we get for all $\delta \le 1$, 
\begin{eqnarray}
\label{DdeltaD}
|D^\delta_z[0,t]-D_z[0,t]| & \le& C\, \sum_{i\ge 0}\,\,  \int_{\tau^\delta_i\wedge t}^{\tau^\delta_{i+1}\wedge t} 
\frac{1}{\|z+\beta_s\|^3\vee 1} \, ds  \le   C \, \int_0^{t}  \frac{1}{\|z+\beta_s\|^3\vee 1} \, ds.
\end{eqnarray}
Taking $z=0$, and combining this with the previous estimates proves part (ii) of the lemma.
\end{proof}

\begin{proof}[\bf Proof of Lemma~\ref{lem-ER}]
We start with the first statement of the lemma. 
Using \eqref{DdeltaD} we obtain 
\begin{align}\label{RdeltaR}
\left|\frac{\mathbb E_{0,z}[R^\delta[0,t]] - \mathbb E_{0,z}[R[0,t]]}{G(0,z)}\right| \   \le \ \estart{\left|\frac{D_z^\delta[0,t] - D_z[0,t]}{G(0,z)}\right|}{} \le\  C \, \E{\int_0^{t} \frac{\norm{z}^2}{\norm{z+\beta_s}^3\vee 1} \,ds}.  
\end{align}
By direct calculations we now get
\begin{align*}
 \E{\int_0^t \frac{\norm{z}^2}{\norm{z+\beta_s}^3\vee 1} \,ds}  \ & = \  \int_{\R^4}\int_0^t \frac{1}{(2\pi s)^2}\frac{\norm{z}^2}{\norm{z+x}^3\vee 1} e^{-\frac{\norm{x}^2}{2s}}\,ds\,dx\\
&=\ \int_{\R^4} \frac{1}{2\pi^2} \frac{\norm{z}^2}{\norm{z+x}^3\vee 1} \frac{1}{\norm{x}^2} e^{-\frac{\norm{x}^2}{2t}}\,dx \\
& \le \  C\, \left(1+ \frac{t}{\norm{z}} + \norm{z} e^{-\frac{\norm{z}^2}{8t}}\right) 
\end{align*} 
and this now completes the proof of the first part of the lemma.

Let us now prove the second part. Note first that 
\begin{align}\label{heur-3}
\estart{R[0,t]}{0,z}=\int_{\R^4} \, G^*(0,z-x) G_t(x)\, dx,
\end{align}
with $G_t(x):=\int_0^t p_s(0,x)\, ds$. Moreover, as we saw in the proof of the previous lemma, 
$G^*(0,z)=(\pi^2/2)\cdot G(0,z)$, when  $\|z\|>1$. 
Therefore for any fixed $x$, $G^*(z-x)/G(z)$ converges to $\pi^2/2$, 
as $\|z\|\to\infty$. Furthermore, using Fubini we can see that $\int G_t(0,x)\, dx=t$. 
We now explain why we can interchange the limit as $z$ goes to infinity and the integral in \eqref{heur-3}.

Indeed, for any $z$ satisfying $\|z\|\ge 1$, let
$F_z=\{x:\ \|z-x\|\le \|z\|/2\}$.
Using standard properties of the Brownian motion, we obtain
for positive constants $C$ and $C'$ independent of $z$, 
\begin{align*}
\begin{split}
\int_{F_z} \frac{G^*(z-x)}{G(z)} G_t(x)\, dx & \le  \ C\,  \|z\|^2\,  \int_{\|x\|\ge
\|z\|/2}\! \, G_t(x)\, dx \, 
= C\,   \|z\|^2 \int_0^t \mathbb P(\|\beta_s\|\ge \|z\|/2 )\, ds \\
& \le \ C2^4 \,  \|z\|^2 \int_0^t \frac{\mathbb E[\|\beta_s\|^4]}{\|z\|^4}\, ds 
\ \le \ C'\,   \frac{t^3}{\|z\|^2}.
\end{split}
\end{align*}
On the other hand on $\R^4\backslash F_z$, the ratio $G^*(z-x)/G(z)$
is upper bounded by a constant, and hence one can apply the dominated convergence theorem.
We conclude that, for any $t>0$, 
$$
\lim_{\|z\|\to \infty} \ \frac{\estart{R[0,t]}{0,z}}{G(0,z)} \ =
\ \frac{\pi^2}{2} \, t.
$$
\end{proof}

\begin{proof}[\bf Proof of Lemma~\ref{lem-D0}]
One has by integrating first with respect to $\til \beta$, and using that $G$ is harmonic on $\R^4\smallsetminus \{0\}$, 
\begin{align*}
\begin{split}
 \E{D_0[0,t]} &= \ \E{\int_0^\infty du\int_0^t ds \,\1
(\|\beta_s-\til{\beta}_u\|\leq 1) }\\
&=\ \int_0^t \E{\int_{\B(0,1)} G(\beta_s-z)\, dz}\, ds \\
&=\ \frac{\pi^2}{2} \, \int_0^t \E{G(\beta_s)\1(\|\beta_s\|>1)} \, ds \, +\,  \mathcal O\left( \int_0^t \mathbb P(\|\beta_s\|\le 1)\, ds\right) \\
&=\ \frac{\pi^2}{2} \, \int_0^t \int_{\|x\|>1} \frac{G(x)}{2\pi^2 s^2} e^{-\frac{\|x\|^2}{2s}} \, dx \, ds \,  + \, \mathcal O(1)\\
&= \ \frac{1}{8 \pi^2} \, \int_{\|x\|>1} \frac{1}{\|x\|^4} e^{-\frac{\|x\|^2}{2t}} \, dx  \, +\, \mathcal O(1), 
\end{split}
\end{align*}
applying Fubini at the last line.  Using now a change of variable 
the last integral is equal to 
\sout{we get }
\begin{eqnarray*}
 \frac{1}{8 \pi^2} \, \int_{\|x\|>1} \frac{1}{\|x\|^4} e^{-\frac{\|x\|^2}{2t}} \, dx& =&  \frac{1}{8\pi^2} \int_1^\infty \frac{2\pi^2\rho^3}{\rho^4}e^{-\frac{\rho^2}{2t}}\, d\rho\ =\  \frac 14 \, \int_{\frac{1}{\sqrt{t}}}^{\infty} \frac{1}{r}e^{-\frac{r^2}{2}}\, dr\\
 &=& \frac 14 \, \int_{\frac{1}{\sqrt{t}}}^{1}\frac{1}{r}\,dr + \mathcal O(1) = \frac{\log t}{8} + \mathcal O(1).
\end{eqnarray*} 
It remains to upper bound the second moment of $D_0[0,t]$. 
Recalling \eqref{Dx}, and by using the Markov property, we get
\begin{eqnarray}\label{D02}
\nonumber \mathbb E[D_0[0,t]^2] &=& \mathbb E\left[\int_0^t \int_0^tG^*(\beta_s)G^*(\beta_{s'})\, ds\, ds'\right] \\
\nonumber &=& 2\int_{0\le s\le s'\le t} \E{G^*(\beta_s)G^*(\beta_{s'})} \, ds\,ds'\\
\nonumber &\le & 2 \int_0^t \, ds\, \mathbb E\left[G^*(\beta_s) \, \mathbb E\left[\int_0^t G^*(\beta_s,\beta_{s'})\, ds' \mid \beta_s\right]\right]\\
&\le & 2 \E{D_0[0,t]} \cdot \left(\sup_{z\in \R^4} \E{D_z[0,t]}\right). 
\end{eqnarray}
Now a simple computation shows that for any $z\in \R^4$ and $t>0$,
$$\mathbb P(\|\beta_t-z\|\le 1) \le \mathbb P(\|\beta_t\|\le 1).$$
Using next that if $\beta$ and $\til \beta$ are two independent standard Brownian motions, then $\beta_u-\til \beta_s$ equals in law $\beta_{u+s}$, for any fixed positive $u$ and $s$, we deduce that also for any $z\in \Z^4$, 
$$\mathbb P_{0,z}(\|\beta_u-\til \beta_s\|\le 1) \ =\ \mathbb P_{0,0}(\|\beta_u-\til \beta_s-z\|\le 1) \ \le\ \mathbb P_{0,0}(\|\beta_u-\til \beta_s\|\le 1),$$
where $\bP_{0,z}$ denotes the law of two independent Brownian motions $\beta$ and $\til \beta$ starting respectively from 
$0$ and $z$. In other terms, one has
$$\E{D_z[0,t]}\ \le \ \E{D_0[0,t]},$$
for all $z\in \R^4$. Together with~\eqref{D02}, this shows that 
$$\mathbb E[D_0[0,t]^2] \ \le \ 2\mathbb E[D_0[0,t]]^2,$$
which concludes the proof, using also the first part of the lemma. 
\end{proof}

\begin{proof}[\bf Proof of Lemma \ref{lem:infint}]
Let 
$$\tau_{k,z}:=\inf\{ s\in [k,k+1]\ :\ \|\til \beta_s -z\| \le 1\}.$$
Note that almost surely, 
\begin{align}\label{mimic-3}
\1\left(\tau_{k,z}\le k+1,\, \sup_{0\le u\le 1}\| \til \beta(\tau_{k,z}+u)-\til \beta(\tau_{k,z})\|\le 1\right)\, \le\,  \int_k^{k+2} \1(\|\til \beta_s-z\| \le 2)\, ds,
\end{align}
just because when the indicator function on the left-hand side equals $1$, we know that $\til \beta$ remains within 
distance at most $2$ from $z$ during a time period of length at least 1. Now, we can use the strong Markov property at time $\tau_{k,z}$ to obtain
\[
\pr{\tau_{k,z}\le k+1,\, \sup_{0\le u\le 1}\| \til \beta(\tau_{k,z}+u)-\til \beta(\tau_{k,z})\|\le 1}\ =\  \pr{\tau_{k,z}\le k+1} \cdot \pr{\sup_{0\le u\le 1}\| \til \beta_u\| \le 1}.
\]
Thus, the lemma follows after taking expectation in \reff{mimic-3}, with $C= 1/\pr{\sup_{0\le u\le 1}\| \til \beta_u\| \le 1}$, which is a positive and finite constant.  
\end{proof}


\subsection{Proof of Lemma~\ref{LD.D0}}\label{subsec:LD.D0}
The idea of the proof is to show that $D_0[0,t]$ is close to a sum of order $\log t$ i.i.d.\ terms with enough finite moments  
(to be more precise we will see that the square root of each of these terms has some finite exponential moment), 
and then apply standard concentration results. This idea was also guiding Lawler's intuition in the discrete setting, as he explains 
in his book \cite{Law91} p.98, in order to understand why $D_0[0,t]$ should be concentrated. 
Since he was not looking for sharper estimates, 
he just showed that the variance of (the analogue in the discrete setting of) $D_0[0,t]$ was of order $\log t$, as its mean. 
But he did it by direct computations, without pushing further this idea of viewing $D_0[0,t]$ as a sum of i.i.d.\ terms. 
Here we will make it more precise (taking advantage of the continuous setting and of the scaling property of the Brownian motion)  
and deduce some better bounds. 
In fact we do not really need the full strength of Lemma~\ref{LD.D0}. However, having just a control of the variance would not be sufficient for the proof; we need at least a good control of the eighth centered moment. Since this is not much more difficult or longer to obtain, 
we prove the stronger result stated in the lemma instead. 
 
First, let us define
the sequence of stopping times $(\tau_i)_{i\ge 0}$ by  
\begin{eqnarray*}
\tau_i : = \inf\{s\ge 0 \ :\ \|\beta_s\|>2^i\},
\end{eqnarray*} 
for all $i\ge 0$. 
Then set for $i\ge 0$, 
\begin{eqnarray*}
Y_i := \int_{\tau_i}^{\tau_{i+1}} G(\beta_s) \, ds, 
\end{eqnarray*}
and for $n\ge 0$, 
$$D_n\ :=\ \sum_{i=0}^n\,  Y_i.$$
Note that in dimension four, for any positive real $\lambda$ and $x\in \R^4$, one has 
$ \lambda ^2 G(\lambda x)= G(x)$. 
Therefore using the scaling property of the Brownian motion, we see that 
the $Y_i$'s are independent and identically distributed. 
The following lemma shows that $Y_0$ has sufficiently small moments, and as a consequence that $D_n$ is concentrated. We postpone its proof. 

\begin{lemma}\label{lem-Y1} There exists a positive constant $\lambda$, such that 
\begin{align*}
\mathbb E[e^{\lambda \sqrt{Y_0}}] \ < \ \infty. 
\end{align*}
As a consequence there exist positive constants $c$ and $C$, such that for all $\epsilon>0$ and $n\ge 1$,  
	\begin{align*}
\pr{|D_n-\mathbb E[D_n]| > \epsilon \, \mathbb E[D_n]}\ \le\ C\,  \exp (  - c \, (\epsilon n)^{1/3} ).
\end{align*}
\end{lemma}
Now we will see that $D_0[0,t]$ is close to $D_{N_t}$, where $N_t$ is defined for all $t>0$, by 
\begin{align*}
N_t=\sup\{i\, :\,  \tau_i\le t\}. 
\end{align*}
Indeed, recall that 
\begin{align*}
D_0[0,t] = \int_0^t G^*(\beta_s)\, ds,
\end{align*}
and that $G^*(z) = G(z)$, whenever $\|z\|>1$. 
Therefore 
\begin{align}\label{sumiid}
D_0[0,t]  \ =\ D_{N_t}  - Z_1(t) - Z_2(t) + Z_3(t), 
\end{align}
with 
$$Z_1(t)=\int_{\tau_0\wedge t}^t \1(\|\beta_s\|\le 1) G(\beta_s)\, ds,\ 
Z_2(t)= \int_t^{\tau_{N_t+1}} G(\beta_s)\, ds, \text{ and } Z_3(t) = \int_0^t \1(\|\beta_s\|\le 1) G^*(\beta_s)\, ds .$$
Since, $G^*$ is bounded on $\B(0,1)$, we see that $Z_3(t) \le Z_3(\infty)\le C\, \ell(A_0)$, for some constant $C>0$, with the notation introduced at the beginning of Section \ref{sec:lemmas}. 
Moreover, by definition $Z_2(t) \le Y_{N_t}$. These bounds  
together with \eqref{localtime} and the next lemma show that $Z_1(t)$, $Z_2(t)$ and $Z_3(t)$ are negligible in \eqref{sumiid}.  

\begin{lemma}\label{lem:YNt}
There exists $\lambda>0$, such that 
$$ \mathbb E[e^{\lambda \sqrt{Z_1(\infty)}}] \ <\ +\infty,$$
and for any $\varepsilon>0$, there exist $c>0$ and $C>0$, such that  
$$\mathbb P(Y_{N_t}\, \ge \, \epsilon \log t)\ \le \ C\exp(-c \, \sqrt{\log t} ).$$ 
Moreover, $\mathbb E[Y_{N_t}] = o(\log t)$. 
\end{lemma}
Let us postpone the proof of this lemma and continue the proof of Lemma~\ref{LD.D0}.

Actually the proof is almost finished. First, all the previous estimates and \eqref{sumiid} show that $D_0[0,t]$ and $D_{N_t}$ have asymptotically the same mean, i.e.
$$\lim_{t\to \infty}\  \frac 1{d(t)} \, \mathbb E[D_{N_t}] \ =\ 1.$$
Moreover, using the strong Markov property at times $\tau_i$, one obtains  
\begin{eqnarray*}
\mathbb E[D_{N_t}] &=& \sum_{i=0}^\infty \mathbb E[Y_i \1(i\le N_t)] \ =\ \sum_{i=0}^\infty \mathbb E[Y_i \1(\tau_i\le t)]\\
&=& \sum_{i=0}^\infty \mathbb E[Y_i] \mathbb P(\tau_i\le t)\ =\ \mathbb E[Y_0] \mathbb E[N_t].
\end{eqnarray*}

Then all that remains to do is to recall that $N_t$ is concentrated. 
Indeed, letting $n_t= \log t/(2\log 2)$, it follows from~\eqref{confinment1} that for any $\epsilon>0$, 
\begin{align}\label{Nt+}
\mathbb P(N_t\ge (1+\epsilon) n_t ) \ =\ \mathbb P\left(\sup_{s\le t} \|\beta_s\|>t^{(1+\epsilon)/2}\right) \ \le\ C\,\exp(-ct^\epsilon),
\end{align}
and it follows from~\eqref{confinment2} that 
\begin{align}\label{Nt-}
\mathbb P(N_t\le (1-\epsilon) n_t ) \ =\ \mathbb P\left(\sup_{s\le t} \|\beta_s\|\le \, t^{(1-\epsilon)/2}\right) \ \le\ C\,\exp(-ct^{\epsilon}),
\end{align}
for some positive constants $c$ and $C$. So for all $\epsilon<1$ we obtain $\E{N_t}\geq (1-\epsilon)n_t$ for all $t$ sufficiently large. 
Therefore,
$$d(t)\ \sim \ \mathbb E[D_{N_t}] \geq  c_0\, (1-\epsilon)n_t,$$
with $c_0=\mathbb E[Y_0]$. Note also that $\E{D_n}=c_0n$, for all $n\ge 0$. 
So now, gathering all previous estimates obtained so far, we deduce   
\begin{eqnarray*}
\mathbb P(D_0[0,t] \ge (1+\epsilon) d(t)) &\le&  \mathbb P\left(D_{\left(1+\frac{\epsilon}{4}\right)n_t} \ge \left(1+\frac{\epsilon}{2}\right) d(t)\right) + 
\mathbb P\left(N_t\ge \left(1+\frac \epsilon 4\right)n_t\right)  \\ &+&\mathbb P\left(Z_3(t) \ge \frac \epsilon 2 d(t)\right)
\le  C\exp (- c(\log t)^{1/3}), 
\end{eqnarray*}
and likewise for the lower bound: 
\begin{eqnarray*}
\mathbb P(D_0[0,t] \le (1-\epsilon) d(t)) &\le&  \mathbb P\left(D_{\left(1-\frac \epsilon 4\right)n_t} \le \left(1-\frac{\epsilon}{2}\right) d(t)\right) + 
\mathbb P\left(N_t\le (1-\frac \epsilon 4)n_t\right) \\&+& 
\mathbb P\left(Z_1(t)+Z_2(t) \ge \frac{\epsilon}{2} d(t)\right)
\le  C\exp (- c(\log t)^{1/3}), 
\end{eqnarray*}
and this concludes the proof of Lemma~\ref{LD.D0}. 
\hfill $\square$

Now to be complete it just remains to prove Lemma~\ref{lem-Y1} and~\ref{lem:YNt}.

\begin{proof}[\bf Proof of Lemma~\ref{lem-Y1}]
We first extend the definition of the $\tau_i$ and $A_i$ to negative indices: 
$$\tau_{-i} : =\inf\{s\ge \tau_0\ : \ \beta_s\in \partial \B(0,2^{-i})\},$$
and 
$$A_{-i} = \B(0,2^{-i+1})\setminus \B(0,2^{-i}),$$
for $i\ge 1$. Then with the notation of Section \ref{sec:lemmas} we get 
$$Y_0\ =\  \int_{\tau_0}^{\tau_1} G(\beta_s) \, ds \ \le\ C\, \left(\sum_{i\ge 1} \1(\tau_{-i+1}<\tau_1) 2^{2i}\ell(A_{-i}) + \tau_1\right). $$
Note that $\tau_1$ has an exponential tail by \eqref{confinment2}, so it suffices to bound the moments of the first sum. 
More precisely it amounts to proving that its $k$-th power is bounded by $C^k \, (k!)^2$. First,  
\begin{align*}
\mathbb E\left[\left(\sum_{i\ge 1} \1(\tau_{-i+1}<\tau_1) 2^{2i}\ell(A_{-i})\right)^k\right] \ =\ \sum_{i_1,\dots,i_k} 4^{\sum_{j=1}^k i_j}\, \E{\prod_{j=1}^k \1(\tau_{-i_j+1}<\tau_1) \ell(A_{-i_j})}.
\end{align*}
Next, by Holder's inequality we get
\begin{align*}\E{\prod_{j=1}^k \1(\tau_{-i_j+1}<\tau_1) \ell(A_{-i_j})}\ 
\leq \ \prod_{j=1}^k \E{\1(\tau_{-i_j+1}<\tau_1) \ell(A_{-i_j})^k}^{1/k}.
\end{align*}
Now by scaling and rotational invariance of the Brownian motion, for any $x\in \partial \B(0,2^{-i+1})$, 
$$\mathbb E_x[\ell(A_{-i})^k] \ =\ 4^{-k(i-1)}\, \mathbb E[\ell(A_{-1})^k].$$
Therefore using this and the strong Markov property, we get  
\begin{align*}
\E{\1(\tau_{-i_j+1}<\tau_1) \ell(A_{-i_j})^k}\ = \ \mathbb P(\tau_{-i_j+1}<\tau_1) \, 4^{-k\, (i_j-1)}\,  \mathbb E[\ell(A_{-1})^k].
\end{align*}
From~\eqref{localtime} we deduce that there is a constant $C>0$, such that 
\begin{align*}
\mathbb E\left[\left(\sum_{i\ge 1} \1(\tau_{-i+1}<\tau_1) 2^{2i}\ell(A_{-i})\right)^k\right]
\ \le & \ C^k \, k! \, \sum_{i_1,\dots,i_k}  \prod_{j=1}^k \mathbb P(\tau_{-i_j+1}<\tau_1)^{1/k}\\ 
=& \ C^k \, k!\, \left(\sum_{i\ge 1} \mathbb P(\tau_{-i+1}<\tau_1)^{1/k}\right)^k\\
\le & \ C^k \, k!\,  \left(\sum_{i\ge 1} \frac{1}{2^{2i/k}}\right)^k \\
\le & \ C^k \, (k!)^2,
\end{align*}	
using \eqref{hit.ball} at the third line. 
This concludes the proof of the first part of the lemma.

Now we prove the second part. Let $\epsilon>0$ be fixed. Since $Y_0$ is integrable, there exists $L\ge 1$, such that  
$\mathbb E[Y_0\1(Y_0>L)] \le  \epsilon /4$. Then using Bernstein's inequality and the first part of the lemma at the third line, we obtain for some positive constants $C$ and $c$,  
\begin{align*}
\begin{split}
\pr{\left| \sum_{i=0}^n(Y_i-\mathbb E[Y_i]) \right|\, >\, \epsilon (n +1)}\ &\le \  \pr{\exists i\le n\, :\,  Y_i>L}+ \pr{ \left| \sum_{i=0}^n(Y_i\1(Y_i<L)-\E{Y_i}) \right|\, > \, \epsilon (n+1)}\\
\le  \ (n+1) \pr{Y_0>L}&+ \pr{\left| \sum_{i=0}^n(Y_i\1(Y_i<L)-\E{Y_i\1(Y_i<L)}) \right|\ > \frac{\epsilon}{2} (n+1)}\\
\le & \ C\left(n\exp(-\lambda\, \sqrt L)+ \exp\left(-c\, \frac{\epsilon^2 n}{  \E{Y_0^2}+ L \epsilon }\right)\right).
\end{split}
\end{align*}
The desired result follows by taking $L= (\epsilon n)^{2/3}$, and $\epsilon n$ large enough.  
\end{proof}

\begin{proof}[\bf Proof of Lemma~\ref{lem:YNt}]
We start with the first part. Exactly as in the proof of Lemma~\ref{lem-Y1}, and using the same notation, one has  
$$Z_1(\infty)\ = \ \int_{\tau_0}^\infty \1(\|\beta_s\|\le 1) G(\beta_s)\, ds \ \le\ C\, \sum_{i\ge 1} \1(\tau_{-i+1}<\infty) 2^{2i}\ell(A_{-i}),$$
and the result follows exactly as in the previous lemma.

Now for the second part, recall the notation introduced at the end of the proof of Lemma~\ref{LD.D0}. Then 
using \eqref{Nt+}, \eqref{Nt-} and Lemma \ref{lem-Y1}, we get 
\begin{eqnarray*}
\mathbb P(Y_{N_t} \ge \epsilon \log t) &\le& \mathbb P(|N_t-n_t| \ge \epsilon \log t) +\mathbb P\left(\exists i\in [n_t-\epsilon \log t,n_t+\epsilon \log t]  \, :\,  Y_i \ge \epsilon \log t\right)\\
&\le & C\, \exp(-c\, t^\epsilon) + 2\epsilon \log t\cdot  \mathbb P(Y_0\ge \epsilon \log t)\\
 &\le & C\, \exp(-c\, t^\epsilon) + C\epsilon (\log t)\, \exp (-c\sqrt {\epsilon \log t}).  
\end{eqnarray*}
Finally we compute the expectation of $Y_{N_t}$ as follows: for any fixed $\epsilon>0$,   
\begin{eqnarray*}
\mathbb E[Y_{N_t}]&=& \sum_{i\ge 0}\mathbb E[ \1(\tau_i\le t< \tau_{i+1}) Y_i]\\
&\le  &\sum_{i\le n_t-\epsilon \log t}\mathbb E[ \1(t<\tau_{i+1}) Y_i] + \sum_{i\ge n_t+ \epsilon \log t} \mathbb E[\1(\tau_i\le t)Y_i] + 2\epsilon (\log t)\mathbb E[Y_0].
\end{eqnarray*}
Then using Cauchy-Schwartz, \eqref{Nt+} and \eqref{Nt-}, we get 
\begin{eqnarray*}
\mathbb E[Y_{N_t}]\ \le\ C\, n_t \exp(-c t^{\epsilon})\E{Y_0^2}^{1/2}  + \E{Y_0}\sum_{j\ge \epsilon \log t} \exp(-c 2^j)  + 2\epsilon (\log t)\mathbb E[Y_0], 
\end{eqnarray*}
and the result follows. 
\end{proof}


\section{Upward Large Deviation}\label{Sec.updev}
Using our estimate on the expected capacity, we obtain
a rough estimate on the upward large deviation, which we use
in the next section when bounding the square of the cross-terms. 
Our estimate
improves a recent inequality of Erhard and Poisat: inequality (5.55)
in the proof of their Lemma 3.7 in \cite{EP}.  
They estimated the probability that the capacity of the sausage exceeds 
by far its mean value and obtained polynomial bounds. 
\begin{proposition}\label{prop-ULD}
There exists a constant $c>0$, such that for any $a>0$, there is $\kappa>0$ satisfying 
\begin{align*}
\pr{\cc{W_1[0,t]}-\E{\cc{W_1[0,t]}}> a\frac{t}{\log t}} \ \le\ 
\exp\left(-c\, a\, t^\kappa\min(1,\frac{a}{\log t})\right).
\end{align*}
Moreover, the inequality holds true for any $a\ge 1$, with $\kappa=1/1000$.  
\end{proposition}
\begin{remark}\label{rem.Lp} \em{
The proposition shows in particular that the process $\left(\frac{\log t}{t} \cp(W_1[0,t]),\, t\ge 2\right)$, 
is bounded in $L^p$, for all $p\ge 1$. }
\end{remark}

\begin{proof}[\bf Proof of Proposition~\ref{prop-ULD}]
Let $a>0$ be fixed. Using that the capacity is subadditive,  one has 
for any $t>0$ and $L\ge 1$, 
\begin{align}\label{ULD-2}
\cc{W_1[0,t]}\ \le\  \sum_{k=0}^{2^L-1} \cc{W_1\left[k\frac{t}{2^L},(k+1)\frac{t}{2^L}\right]}.
\end{align}
To simplify notation now, we write 
$$X= \cc{W_1[0,t]}, \quad \text{and}\quad X_k = \cc{W_1\left[k\frac{t}{2^L},(k+1)\frac{t}{2^L}\right]},\quad \text{for }k\ge 0.$$
Note that the $(X_k)$ are independent and identically distributed. 
Now choose $L$ such that $2^L=[t^\kappa]$, with $\kappa<1$, some positive constant to be fixed later. 
Then for $t$ large enough, Proposition~\ref{prop-expected} gives  
\begin{align*}
\E{X}\ge 4\pi^2(1-2^{-10}a) \frac{t}{\log t},\quad
\text{and}\quad \E{X_1}\le 4\pi^2(1+2^{-10}a) 
\frac{t/2^L}{\log(t/2^L)}. 
\end{align*}
Plugging this into~\reff{ULD-2} we obtain  
\begin{align*}
X- \E{X} \le \sum_{k=0}^{2^L-1}  (X_k-\E{X_k}) +4\pi^2\frac{t}{\log t}
\Big(\frac{(1+2^{-10}a)}{1-\log(2^L)/\log t} -(1-2^{-10}a)\Big).
\end{align*}
Now when $a \le 1$, by choosing $\kappa$ small enough (depending on $a$), one can make 
the last term above smaller than $ a t/(2\log t)$, and when $a\ge 1$, it is easy to check that this is also true with $\kappa=1/1000$. 
Thus for this choice of $\kappa$, 
\begin{align}
\label{LDcap.1}
\pr{X-\E{X} \ge a \frac{t}{\log t}} \ \le\ \pr{ \sum_{k=0}^{2^L-1} (X_k-\E{X_k})\ge \frac{a}{2} 
\frac{t}{\log t}}.
\end{align}
Now we claim that $X_1/(t/2^L)$ has a finite exponential moment. Indeed, thanks to 
Lemma~\ref{lem.epsilon}, it suffices to compute the moments of the volume of a Wiener sausage. But this is easily obtained, using a similar argument as for the local time of balls, see~\eqref{localtime}. To be more precise, for $z\in\R^4$, set 
$$\sigma_z:=\inf\{s\ge 0\ :\ \|\beta_s-z\|\le 1\}.$$
Then  for any $t\ge 1$ and $k\ge 1$, one has using the Markov property 
 \begin{eqnarray*}
\mathbb E[|W_1[0,t]|^k] & =&  \int \dots \int \bP(\sigma_{z_1}\le t,\dots,\sigma_{z_k}\le t)\, dz_1\dots dz_k\\ 
&=& k! \int\dots\int \bP(\sigma_{z_1}\le \dots \le \sigma_{z_k}\le t)\, dz_1\dots dz_k\\
&\le & k!\,  \mathbb E[|W_1[0,t]|]^k.
\end{eqnarray*}
Now recall a classical result of Kesten, Spitzer, and Whitman on the 
volume of the Wiener sausage, 
(see e.g. \cite{LG88} or \cite{LG.90} and references therein).
\begin{equation*}
\lim_{t\to \infty} \frac{1}{t} \cdot \mathbb E[|W_1(0,t)|] = \cp(\mathcal B(0,1)) = 2\pi^2. 
\end{equation*}
As a consequence, for some constant $C$, we have
$\mathbb E[|W_1[0,t]|^k]\le  C^k k! t^k$, and
there exists $\lambda_0>0$, such that 
\begin{eqnarray}
\label{LDcap.2}
\sup_{t\ge 1}\, \mathbb E\left[\exp\left(\lambda_0 \, \frac{\cp(W_1[0,t])}{t}\right)\right]\ < \ +\infty.
\end{eqnarray}

Now from \eqref{LDcap.1} and \eqref{LDcap.2} it is quite standard to deduce the result of the proposition. 
But let us give some details for the reader's convenience.  
First, using a Taylor expansion, one has for any $x\in \R$, and any integer $n\ge 0$, 
\[
\left|e^x-\sum_{i=0}^n \frac{x^i}{i!}\right|\ \le\ 
e^{|x|} \frac{|x|^{n+1}}{(n+1)!}.
\]
Applying this with $n=2$, shows that for any $\lambda\ge 0$, and any nonnegative random
variable $Y$ with finite mean,
\begin{align*}
\left|e^{\lambda (Y-\mathbb{E}[Y]) }-\sum_{i=0}^2 \frac{\lambda^i (Y-\E{Y})^i}{i!}\right|\ \le\ 
\frac{\lambda^3}{3! }\,  |Y-\E{Y}|^{3}\, e^{\lambda |Y-\mathbb{E}[Y]|}.
\end{align*}
Therefore, if we assume in addition that $\mathbb E[e^Y]$ is finite and that $\lambda\le 1/2$, we obtain 
\begin{align*}
\mathbb E[e^{\lambda (Y-\mathbb{E}[Y])}]\ \le\  1+ \frac{\lambda^2}{2}\mathbb E[(Y-\E{Y})^2]+
C_1 \lambda^3\ \le \ e^{C_2 \lambda^2},
\end{align*}
for some constants $C_1$ and $C_2$ (that only depend on $\mathbb E[e^Y]$). 
Now we apply the previous bound to $Y=\lambda_0 X_0/(t/2^L)$, with $\lambda_0$ as in \eqref{LDcap.2}. 
Then using Chebychev's exponential inequality, we get for any $\lambda\in [0,1/2]$,  
\begin{align*}
\begin{split}
\pr{ \sum_{k=0}^{2^L-1} \frac{(X_k-\E{X_k})}{t/2^L} \, \ge\,  \frac{a}{2} \frac{2^L}{\log t} }\ \le &\ \exp\big( -\frac{\lambda \lambda_0a}{2\log t} 2^L\big)
\prod_{k=0}^{2^L-1} \mathbb E\left[\exp\left(\lambda \lambda_0\frac{X_k-\E{X_k}}{t/2^L}\right)\right] \\
\le &\ \exp\left( -\left(\frac{\lambda \lambda_0a}{2\log t}-C_2\lambda^2\right)2^L
\right),
\end{split}
\end{align*}
and the result follows by optimizing in $\lambda$. 
\end{proof}


\section{Intersection of Sausages and Cross-terms}
\label{subsec.intersection}
\subsection{Intersection of Wiener sausages}
Our aim in this Section is to obtain some bounds on the probability 
of intersection of two Wiener sausages.
Then, in the next section, we apply these 
results to bound the second moment of the cross-term in the decomposition \eqref{chi.r} of the capacity of two Wiener sausages.

We consider two independent Brownian motions $(\beta_t,t\ge 0)$ and $(\til \beta_t,t\ge 0)$ starting respectively from $0$ and $z$, and denote their corresponding Wiener sausages by $W$ and $\widetilde W$.  
We estimate the probability that $W_{1/2}[0,t]$ 
intersects $\widetilde W_{1/2}[0,\infty)$,  
when $\|z\|$ is of order $\sqrt t$ up to logarithmic factors.

Such estimates have a long history in probability. Let us mention three occurrences of closely related estimates, which are however not
enough to deduce ours. Aizenman in~\cite{A} obtained a bound for
the Laplace transform integrated over space. Pemantle, Peres and Shapiro \cite{PPS} obtained that for
any $z\in \R^4$ and $t$ large enough, almost surely 
\[
\frac{c t}{\log t} \inf_{y\in \beta[0,t]}\|z-y\|^{-2}\le
\bP_{0,z}\left(W_{1/2}[0,t]\cap\widetilde W_{1/2}[0,\infty)
\not=\emptyset\ \Big|\ \beta\right)
\le \frac{C t}{\log t} \sup_{y\in \beta[0,t]}\|z-y\|^{-2}.
\]
Lawler has obtained also similar results in the discrete setting
for random walks. Finally, our result reads as follows.

\begin{proposition}\label{prop.hit}
For any $\alpha>0$, there exist positive constants $C$ and $t_0$, such that for all $t>t_0$ and $z\in \R^4$, 
with $ t/(\log t)^\alpha \le \|z\|^2 \le  t \cdot(\log t)^\alpha$,  
\begin{equation}
\label{hit-2wieners}
\prstart{W_{1/2}[0,t]\cap\widetilde W_{1/2}[0,\infty)\not=\emptyset}{0,z} \ \le\ C \cdot\left(1\wedge \frac{t}{\|z\|^2}\right) \cdot \frac{(\log \log t)^2}{\log t}. 
\end{equation}
\end{proposition}
We divide the proof of Proposition~\ref{prop.hit} into two lemmas. The first one deals with $\|z\|$ {\it large}.
\begin{lemma}
\label{lem.hit.1}
For any $\alpha>0$, there exist positive constants $C$ and $t_0$, such that for all $t>t_0$ and all $z\in \R^4$ nonzero, 
with $\|z\| \le \sqrt t \cdot(\log t)^\alpha$,  
\begin{equation}
\label{hitsaus}
\bP_{0,z}\left(W_{1/2}[0,t]\cap\widetilde W_{1/2}[0,\infty)\not=\emptyset \right) \ \le\ C \cdot \frac {t }{\|z\|^2}\cdot \frac{\log \log t}{\log t}. 
\end{equation}
\end{lemma}
The second lemma improves on Lemma~\ref{lem.hit.1} in the region $\|z\|$ {\it small}.
\begin{lemma}
\label{lem.hit.2}
For any $\alpha>0$, there exist positive constants $C$ and $t_0$, 
such that for all $t>t_0$ and all $z\in \R^4$, with $t\cdot (\log t)^{-\alpha}\le \|z\|^2 \le t $,  
\begin{equation}
\label{hitsaus2}
\bP_{0,z}\left(W_{1/2}[0,t]\cap\widetilde W_{1/2}[0,\infty)\not=\emptyset \right) \ \le\ C\cdot \frac{(\log \log t)^2}{\log t}. 
\end{equation}
\end{lemma}

\noindent{\bf Proof of Lemma~\ref{lem.hit.1}.}
Let $r:= \sqrt{t/\log t}$. Assume that $\|z\|>r$, otherwise there is nothing to prove. 
Using \eqref{hit.ball}, we see that estimating~\reff{hitsaus} amounts to bounding the term
$$
\bP_{0,z}\left(W_{1/2}[0,t]\cap\widetilde W_{1/2}[0,\infty)\not=\emptyset ,\, W_1[0,t]\cap \mathcal B(z,r) =\emptyset\right).
$$ 
Using now Proposition \ref{prop-ULD}, we see that it suffices to bound the term 
$$
\prstart{W_{1/2}[0,t]\cap\widetilde W_{1/2}[0,\infty)\not=\emptyset ,\, d(z,W_1[0,t])\ge r,\, \cp(W_1[0,t]) \le 8\pi^2 \frac t{\log t}}{0,z}.$$
By first conditioning on
$W_1[0,t]$, and then applying Lemma~\ref{cond.hit},
 we deduce that the latter display is bounded, up to a constant factor, by 
$$\mathbb  E\left[\frac{\mathbf 1(d(z,W_1[0,t]) \ge r)}{d(z,W_1[0,t])^2}\right]\cdot \frac{t}{\log t}.$$
Furthermore, on the event $\{d(z,W_1[0,t])\ge r\}$, for $t$ sufficiently large we have 
$$
\frac{1}{2} d(z,\beta[0,t])\le d(z,\beta[0,t])-1\le d(z,W_1[0,t])\le d(z,\beta[0,t]),
$$
with $\beta[0,t]$ the trace of $\beta$ on the time interval $[0,t]$. 
Now by using again \eqref{hit.ball} and the bound $\|z\|\le \sqrt t (\log t)^\alpha$, we get for some constant $C$ independent of $z$,
\begin{eqnarray*}
\E{\frac{\mathbf 1(d(z,\beta[0,t]) \ge r)}{d(z,\beta[0,t])^2}} &=& 2\int_{1/\|z\|}^{1/r} u\cdot  \bP\left(d(z,\beta[0,t])\le 1/u\right) \, du\\
&\le& C \frac{\log (\|z\|/r)}{ \|z\|^2} \le 
C(\alpha+\frac{1}{2})\frac{\log \log t}{ \|z\|^2},
\end{eqnarray*}
which concludes the proof. 
\qed

\noindent{\bf Proof of Lemma~\ref{lem.hit.2}.}
Set $t_1=0,\ t_2=\|z\|^2$ and for $k\ge 3$, denote $t_k=2t_{k-1}$. Let $K$ be the smallest integer such that $2^{K-1}\ge (\log t)^\alpha$. In particular $t\le 2^{K-1} \|z\|^2 = t_{K+1}$ by hypothesis. Then,
$$
\bP_{0,z}\left(W_{1/2}[0,t]\cap\widetilde W_{1/2}[0,\infty)\not=\emptyset\right)
 \ \le \ \sum_{k=1}^K \bP_{0,z}\left(W_{1/2}[t_k,t_{k+1}]\cap\widetilde W_{1/2}[0,\infty)\not=\emptyset \right).
 $$
We now bound each term of the sum on the right hand side. 
The first one (corresponding to $k=1$) is bounded using directly Lemma \ref{lem.hit.1}: for some positive constant $C$, 
$$
\bP_{0,z}\left(W_{1/2}[0,\|z\|^2]\cap\widetilde W_{1/2}[0,\infty)\not=\emptyset\right) \ \le\ C \cdot \frac{\log \log t}{\log t}.
$$
Now for the other terms, we first observe that for some 
positive constants $\kappa$, $C$, and $C'$, 
\begin{eqnarray}
\label{loc.martingale}
 \E{\frac{1}{\|\beta_{t_k}-z\|^2}} \ \le C\ \frac{1}{t_k^2}\cdot  \int \frac{1}{\|z-x\|^2} e^{-\kappa\cdot \|x\|^2/t_k}\, dx \ \le \  \frac{C'}{t_k}.
\end{eqnarray} 
Then, we obtain, for some positive constant $C$, 
\begin{eqnarray*}
\bP_{0,z}\left(W_{1/2}[t_k,t_{k+1}]\cap\widetilde W_{1/2}[0,\infty)\not=\emptyset \right) &\le & 
\E{\bP_{0,z-\beta_{t_k}}\left(W_{1/2}[0,t_{k+1}-t_k]\cap\widetilde W_{1/2}[0,\infty)\not=\emptyset\right)}\\
&\le & C\, \E{\frac{1}{\|\beta_{t_k}-z\|^2}}\cdot \frac{t_{k} \cdot \log \log t}{\log t}\ \le\  C \cdot \frac{\log \log t}{\log t} ,
\end{eqnarray*}
using again Lemma \ref{lem.hit.1} and \eqref{confinment1} at the second line and \eqref{loc.martingale} for the third inequality.  
We conclude the proof recalling that $K$ is of order $\log \log t$.
\qed

We now give the proof of Proposition~\ref{lem.hit.3}. 

\begin{proof}[\bf Proof of Proposition~\ref{lem.hit.3}]
Define the stopping times
$$
\sigma:= \inf\{s\ : \ W_{1}[0,s]\cap \gamma[0,\infty)\neq \emptyset\},
\quad\text{and}\quad
\widetilde \sigma:= \inf\{s\ : \ W_{1}[0,s]\cap \widetilde \gamma[0,\infty)\neq \emptyset\}.
$$
Note that 
$$
\bP_{0,z,z'}\left(W_1[0,t] \cap \gamma[0,\infty)\not=\emptyset ,\  W_1[0,t]
 \cap\widetilde \gamma[0,\infty)\not=\emptyset\right) = \bP_{0,z,z'}(\sigma<\widetilde \sigma\le t) + \bP_{0,z,z'}(\widetilde \sigma< \sigma\le t).
$$
By symmetry, we only need to deal with $ \bP_{0,z,z'}(\sigma<\widetilde \sigma\le t)$.
Now conditionally on $\gamma$, $\sigma$ is a stopping time for $\beta$. In particular, conditionally on $\sigma$ and $\beta_\sigma$, $W_1[\sigma, t]$ is equal in law to $\beta_\sigma +  W_1'[0,t-\sigma]$, with $W'$ a Wiener sausage, independent of everything else. 
Therefore 
\begin{eqnarray*}
\bP_{0,z,z'}(\sigma<\widetilde \sigma\le t) &\le &  \mathbb E_{0,z}\left[\mathbf{1}(\sigma \le t)\,  \bP_{0,z,z'}(\sigma < \widetilde \sigma \le t\mid \sigma,\, \gamma, \beta_\sigma)\right]\\
&\le &  \mathbb E_{0,z}\left[\mathbf{1}(\sigma \le t)\,  
\bP_{0,z'-\beta_\sigma}(W'_1[0,t-\sigma]\cap \til{\gamma}[0,\infty)\neq \emptyset\mid \sigma)\right]\\
&\le &  \mathbb E_{0,z}\left[\mathbf{1}(\sigma \le t)\,  
\bP_{0,z'-\beta_\sigma}(W'_1[0,t]\cap \til{\gamma}[0,\infty)\neq \emptyset)\right].
\end{eqnarray*}
To simplify notation, write $D=\|z'-\beta_\sigma\|$. Note that one can assume $D> \sqrt t \cdot (\log t)^{- 3\alpha-1}$, since by using \eqref{hit.ball} and the hypothesis on $\|z'\|$ we have 
$$\bP\left(\sigma\leq t, D\le  \sqrt t \cdot (\log t)^{- 3\alpha-1}\right)\  \le\  \frac{t}{\|z'\|^2\cdot (\log t)^{6\alpha + 2}}\ \le \ (\log t)^{-4\alpha - 2},$$
and the right hand side in \eqref{hit.zz'} is always larger than $(\log t)^{-4\alpha - 2}$ by the hypothesis on $z$ and $z'$. 
Then by applying Proposition~\ref{prop.hit} 
we get for positive constants $C_1$ and $C_2$, 
\[
\begin{split}
&\mathbb E_{0,z}
\ \left[\mathbf{1}\left(\sigma \le t,\, D>\frac{\sqrt t}{(\log t)^{3\alpha + 1}}\right)\, 
 \bP_{0,z'-\beta_\sigma}\left(W'_1[0,t]\cap \til{\gamma}[0,\infty)\neq \emptyset\right)\right] \\
&  \le C_1\  \mathbb E_{0,z} \left[\mathbf{1}(\sigma \le t)\, \left(1\wedge \frac{t}{D^2}\right)\right] \cdot  \frac{(\log \log t)^2}{\log t}\\
 & \le C_1\ \bP_{0,z} \left(W_1[0,t]\cap \til{\gamma}[0,\infty)\neq \emptyset\ \right)\cdot \left(1\wedge \frac{16 t}{\|z'\|^2}\right) \cdot  \frac{(\log  \log t)^2}{\log t}+C_1 \bP_{0,z} \left(\sigma \le t, \, D\le \frac{\|z' \|}{4}\right) \cdot  \frac{(\log \log t)^2}{\log t}\\
& \le\ C_2 \frac{(\log \log t)^4}{(\log t)^2}\cdot\left(1\wedge \frac{t}{\|z\|^2}\right) \cdot\left(1\wedge \frac{t}{\|z'\|^2}\right) +
C_1 \bP_{0,z} \left(\sigma \le t, \, D\le \frac{\|z'\|}{4}\right) \cdot  \frac{(\log \log t)^2}{\log t}.
\end{split}
\]
Now define 
\begin{equation*}
\tau_{z,z'}:= \left\{ 
\begin{array}{ll} 
\inf\{s\ :\ \beta_s\in \mathcal B(z',\|z'\|/4)\} & \text{if }\|z-z'\|>\|z'\|/2\\
\inf\{s\ :\ \beta_s \in \mathcal B(z',3\|z'\|/4)\} & \text{if } \|z-z'\|\le \|z'\|/2.
\end{array}
\right.  
\end{equation*}
Note that by construction $\|z-\beta_{\tau_{z,z'}}\|\ge \max(\|z-z'\|,\|z'\|)/4$, and that on the event $\{D\le \|z'\|/4\}$, one has $\sigma \geq  \tau_{z,z'}$. 
Therefore by conditioning first on $\tau_{z,z'}$ and the position of $\beta$ at this time, and then by using Proposition~\ref{prop.hit}, 
we obtain  for some positive constants $\kappa$, $C_3$ and $C_4$, 
\begin{eqnarray*}
\bP_{0,z} \left(\sigma \le t, \, D\le \|z'\|/4\right) &\le & \bP_{0,z}\left(\tau_{z,z'}\le \sigma \le t\right) \\ 
& \le& C_3\left(1\wedge \frac{t}{\|z-z'\|^2}\right) \cdot  \frac{(\log \log t)^2}{\log t} \cdot \bP(\tau_{z,z'} \le t)\\
&\le&  C_3\left(1\wedge \frac{t}{\|z-z'\|^2}\right) \cdot  \frac{(\log \log t)^2}{\log t}\cdot e^{-\kappa \cdot \|z'\|^2/t} \\
& \le& C_4 \left(1\wedge \frac{t}{\|z\|^2}\right) \left(1\wedge \frac{t}{\|z'\|^2}\right)\cdot  \frac{(\log \log t)^2}{\log t},
\end{eqnarray*}
where we used \eqref{confinment1} in 
the third line and considering two cases to obtain
the last inequality: $\|z'\|\geq \|z\|/2$, in which case we bound the exponential term by the product and $\|z'\|<\|z\|/2$, in which case using the triangle inequality gives $\|z-z'\|\geq \|z\|/2$. This concludes the proof. 
\end{proof}

\subsection{A second moment estimate}\label{sec:secondmoment}
Here we apply the results of the previous subsection to bound the second moment of the cross-term $\chi$ 
from the decomposition \eqref{decomposition.cap}. 
Recall that for any compact sets $A$ and $B$ with $A\cup B\subset \B(0,r)$,
we have defined
\[
\chi_r(A,B) = 2\pi^2 \, r^2 \cdot \frac 1{|\partial \B(0,r)|}
\int_{\partial \B(0,r)} (\bP_z[H_A<H_B<\infty] + \bP_z[H_B<H_A<\infty]) \, dz,
\]

\begin{proposition}\label{prop.chi.Wiener}
Let $\beta$ and $\widetilde \beta$ be two independent Brownian motions 
and let $W$ and $\widetilde W$ be their corresponding Wiener sausages. 
Then, there is a constant $C$ such that for any $t>0$,
with $r(t) = \sqrt t \cdot \log t$,
\begin{equation}\label{key-cross}
\E{\chi^2_{r(t)}(W_1[0,t],\widetilde W_1[0,t])\,
\mathbf{1}\big(W_1[0,t] \cup \widetilde W_1[0,t]\subset\B(0,r(t))\big)}\ 
\le \ C\, \frac{t^2(\log\log t)^8}{(\log t)^4 }.
\end{equation}
\end{proposition} 

\begin{proof}[\bf Proof]  
For any compact sets $A$ and $B$ and any $r$ 
such that $A\cup B\subset \B(0,r)$,
we bound $\chi_r(A,B)^2$ as follows. 
For some constant $C>0$,
\begin{eqnarray}
\label{chir2}
&&\nonumber \chi_{r}(A,B)^2 \ \le \ C\, \frac {r^4}
{|\partial \B(0,r)|^2} \int_{\partial \B(0,r)\times \partial \B(0,r)} 
\left( \bP_{z,z'}(H_A<H_B<\infty,\, \widetilde H_A<\widetilde H_B<\infty)\right.\\\nonumber &+&\left. \bP_{z,z'}(H_B<H_A<\infty,\, \widetilde H_B<\widetilde H_A<\infty)
  + \bP_{z,z'}(H_A<H_B<\infty,\, \widetilde H_B<\widetilde H_A<\infty)\right.\\ &+&\left. \bP_{z,z'}(H_B<H_A<\infty,\, \widetilde H_A<\widetilde H_B<\infty) \right) \, dz\, dz',
\end{eqnarray}
where $H$ and $\widetilde H$ refer to the hitting times of two independent Brownian motions $\gamma$ and $\widetilde \gamma$ starting respectively from $z$ and $z'$ in $\partial \B(0,r)$. 
To simplify notation, let $A=W_1[0,t]$, $B=\widetilde W_1[0,t]$, and $r=r(t)$.
By using \eqref{hit.ball}, we obtain
\begin{eqnarray*}
&& \bP_{z,z'}(H_A<H_B<\infty,\, \widetilde H_A<\widetilde H_B<\infty) \\
&=& \bP_{z,z'}\left(H_A<H_B<\infty, \widetilde H_A<\widetilde H_B<\infty,  H_{\mathcal B(0,\frac{\sqrt t}{(\log t)^3})}=\infty,  \widetilde H_{\mathcal B(0,\frac{\sqrt t}{(\log t)^3})}=\infty\right)+\mathcal O\left(\frac{1}{(\log t)^8}\right).
\end{eqnarray*}
Now, to bound the probability on the right-hand side,
we use the Markov property at times $H_A$ and $\widetilde H_A$ for $\gamma$ and $\widetilde \gamma$ 
respectively. We then have using Lemma \ref{lem.hit.3} twice, for some
constant $C$
\begin{eqnarray}
\label{term1}
\nonumber \bP_{z,z'}(H_A<H_B<\infty,\, \widetilde H_A<\widetilde H_B<\infty)  &
\le& C\, \bP_{z,z'}\left(H_A<\infty,\, \widetilde H_A<\infty\right)\cdot \frac{(\log \log t)^4}{(\log t)^2}\, +\mathcal O\big( \, (\log t)^{-8} \big) \\
\nonumber &\le
& C\left(1\wedge \frac{t}{\|z'\|^2}\right)\cdot\left(1\wedge \frac{t}{\|z\|^2}\right) 
\frac{(\log \log t)^8}{(\log t)^4}+\mathcal O\big( \, (\log t)^{-8} \big)\\
&=&\mathcal O\left( \frac{(\log \log t)^8}{(\log t)^8}\right).
\end{eqnarray}
By symmetry, we get as well 
\begin{equation}
\label{term2}
\bP_{z,z'}(H_B<H_A<\infty,\, \widetilde H_B<\widetilde H_A<\infty)=
\mathcal O\left( \frac{(\log \log t)^8}{(\log t)^8}\right).
\end{equation}
Now to bound the last two terms in \eqref{chir2}, we can first condition on $A=W_1[0,t]$ and $B=\widetilde W_1[0,t]$, and then using the inequality $ab\le a^2 + b^2$ for $a,b>0$, together with \eqref{term1} and \eqref{term2}, this gives   
\begin{equation}\label{term3}
\begin{split}
\bP_{z,z'}(H_A<H_B<\infty,\, &\widetilde H_B<\widetilde H_A<\infty)\le \ \bP_{z,z}(H_A<H_B<\infty,\, \widetilde H_A<\widetilde H_B<\infty) \\
& \ \ \ + \bP_{z',z'}(H_B<H_A<\infty,\, \widetilde H_B<\widetilde H_A<\infty) 
=\mathcal O\left( \frac{(\log \log t)^8}{(\log t)^8}\right).
\end{split}
\end{equation}
By symmetry it also gives  
\begin{equation}
\label{term4}
\bP_{z,z'}\left(H_{\widetilde W_1[0,t]}<H_{W_1[0,t]}<\infty,\, \widetilde H_{W_1[0,t]}<\widetilde H_{\widetilde W_1[0,t]}<\infty\right)=
\mathcal O\left( \frac{(\log \log t)^8}{(\log t)^8}\right).
\end{equation}
Then the proof follows from \eqref{chir2}, \eqref{term1}, \eqref{term2}, \eqref{term3}, and \eqref{term4}. 
\end{proof}


\section{Proof of Theorem \ref{theo.cap.wiener}}\label{sec:finalproof}
The proof of the strong law of large number has four elementary
steps: (i) the representation formula \eqref{cap.formula} of the capacity of the sausage in terms
of a probability of intersection of two sausages, (ii) a decomposition
formula as we divide the time period into two equal periods, and iterate
the latter steps enough times (iii) an estimate of the variance of
dominant terms of the decomposition, 
(iv) Borel-Cantelli's Lemma allows us to
conclude along a subsequence, and the monotony of the capacity which
yields the asymptotics along all sequence.

Since all the technicalities have been dealt before, we present
a streamlined proof. We only give the proof when the radius of the sausage is equal to one, as the same proof applies for any radius.   

\paragraph{The decomposition.} We let $r=r(t)=\sqrt t \cdot \log t$. 
When dealing with
the random set $W_1[0,t]$, \eqref{decomposition.cap}
holds only on the event $\{W_1[0,t]\subset \B(0,r)\}$, and yields
\begin{align*}
\cc{W_1[0,t]}=\cc{W_1\left[0,\frac{t}{2}\right]}+\cc{W_1\left[\frac{t}{2},t\right]}
-\chi_r &\left(W_1\left[0,\frac{t}{2}\right],W_1\left[\frac{t}{2},t\right]\right)
\\
&-\epsilon_r\left(W_1\left[0,\frac{t}{2}\right],W_1\left[\frac{t}{2},t\right]\right).
\end{align*}
What is crucial here is that $\cp(W_1[0,\frac{t}{2}])$ and
$\cp(W_1[\frac{t}{2},t])$ are independent.
We iterate the previous decomposition $L$ times and center it,
to obtain (with the notation $\overline X=X-\mathbb E[X]$), on the event
$\{W_1[0,t]\subset \B(0,r)\}$,
\begin{equation}\label{legall-2}
\overline{\cp(W_1[0,t])}=\overline{S(t,L)}-\overline{\Xi(t,L,r)}
-\overline{\Upsilon(t,L,r)},
\end{equation}
where $S(t,L)$ is a sum of $2^L$ i.i.d.\ terms distributed
as $\cp(W_1[0,t/2^L])$, where
\begin{equation}\label{legall-3}
\Xi(t,L,r)=\sum_{\ell=1}^L\sum_{i=1}^{2^{\ell-1}} 
\chi_r\Big(W_1\left[\frac{2i-2}{2^{\ell}}t,\frac{2i-1}{2^{\ell}}t\right],
W_1\left[\frac{2i-1}{2^{\ell}}t,\frac{2i}{2^{\ell}}t\right] \Big),
\end{equation}
and
\begin{equation}\label{legall-4}
\Upsilon(t,L,r)=\sum_{\ell=1}^L\sum_{i=1}^{2^{\ell-1}} 
\epsilon_r\Big(W_1\left[\frac{2i-2}{2^{\ell}}t,\frac{2i-1}{2^{\ell}}t\right],
W_1\left[\frac{2i-1}{2^{\ell}}t,\frac{2i}{2^{\ell}}t\right]\Big).
\end{equation}
In both \reff{legall-3} and \reff{legall-4}, the second sum (with $\ell$ fixed)
is made of independent terms.
\paragraph{Variance Estimates.}
We choose $L$ such that $(\log t)^4\le 2^L\le 2(\log t)^4$,
so that $L$ is of order $\log \log t$. 
Let now  $\epsilon>0$ be gixed. By \eqref{confinment1} and Chebychev's inequality, 
for $t$ large enough, 
\begin{eqnarray}\label{cheb-1}
\nonumber \bP\big( |\overline{\cp(W_1[0,t])}| >  \epsilon \frac{t}{\log t}\big)
&\le& \bP\left( W_1[0,t]\not\subset \B(0,r)\right)+
\bP\left( |\overline{\Upsilon(t,L,r)}|>\frac{\epsilon}{2} \frac{t}{\log t}
\right)\\
\nonumber && \quad+\,  \bP\left( |\overline{S(t,L)}-\overline{\Xi(t,L,r)}| 
>\frac{\epsilon}{2} \frac{t}{\log t}\right)\\
\nonumber &\le & e^{-c(\log t)^2}+\bP\left( |\overline{\Upsilon(t,L,r)}|
>\frac{\epsilon}{2} \frac{t}{\log t}\right)\\
&&\quad +\, 8(\log t)^2 
\frac{\var(S(t,L))+\var(\Xi(t,L,r))}{\epsilon^2 t^2}. 
\end{eqnarray}
Then we use the triangle inequality for the $L^2$-norm and the Cauchy-Schwarz 
inequality, as well as Proposition \ref{prop.chi.Wiener}, to obtain  
\begin{equation}\label{chi.ell}
\var(\Xi(t,L,r))\le C L \cdot \sum_{\ell= 1}^L 2^{\ell-1} 
\frac{t^2 \cdot (\log \log t)^8}{2^{2\ell}(\log t)^{4}}
\le Ct^2\cdot \frac{(\log \log t)^9}{(\log t)^{4}}.
\end{equation}  
To deal with $\var(S(t,L))$, we can use Proposition \ref{prop-ULD} which gives
a constant $C>0$, such that for any $t\ge 2$
$$\mathbb E[\cp(W_1[0,t])^2]\ \le\  C\, \frac{t^2}{(\log t)^2},$$ 
Thus there exists a  
constant $C'>0$, such that for $t$ large enough, 
\begin{equation}\label{cheb-2}
\var(S_L(t))\ \le\ C' \, 2^L \frac{(t/2^L)^2}{\log^2(t/2^L)}\ \le\ 2C'\, 
\frac{t^2}{(\log t)^6}.
\end{equation}
The term $\Upsilon$ is controlled by invoking 
Lemma \ref{lem.epsilon}, and using that $\varepsilon_r(A,B)\le \cp(A\cap B)$. We deduce 
\begin{equation}\label{cheb-2bis}
\var(\Upsilon(t,L,r)) \ = \ \mathcal O(2^{2L} (\log t)^2) \ 
=\ \mathcal O((\log t)^{10}),
\end{equation}
so that
\begin{equation}\label{cheb-3}
\bP\left( |\overline{\Upsilon(t,L,r)}|>\frac{\epsilon}{2} \frac{t}{\log t}
\right) =\ \mathcal O\left( \frac{(\log t)^{12}}{t^2}\right).
\end{equation}
Plugging \eqref{chi.ell} \reff{cheb-2} and \eqref{cheb-3} into \reff{cheb-1}, 
we obtain
\begin{equation*}
\bP\left( \left|\cp(W_1[0,t])-\E{\cp(W_1[0,t])}\right| \ge \varepsilon\frac{t}{\log t}\right)\, 
=\, \mathcal O\left(\frac{(\log \log t)^9}{(\log t)^2}\right).
\end{equation*}

\paragraph{From Subsequences to SLLN.}
Consider the sequence $a_n = \exp(n^{3/4})$,
satisfying that $a_{n+1}-a_n$ 
goes to infinity but $a_{n+1}-a_n=o(a_n)$.
Since the previous bound holds for all $\varepsilon>0$, by using Borel-Cantelli's lemma and Proposition~\ref{prop-expected}, we deduce that a.s. 
\begin{equation}\label{asconv1}
\lim_{n\to \infty} \, \frac{\cp(W_1[0,a_n])}{\E{\cp(W_1[0,a_n])}} \ = \ 1.
\end{equation}
Let now $t>0$, and choose $n=n(t)>0$, so that $a_n\le t<a_{n+1}$. Using that the map $t\mapsto \cp(W_1[0,t])$ is a.s.\ nondecreasing (since for any sets $A\subset B$, one has $\cp(A)\le \cp(B)$), we can write 
\begin{equation}\label{asconv}
\frac{\cp(W_1[0,a_n])}{\E{\cp(W_1[0,a_{n+1}])}}\, \le \, \frac{\cp(W_1[0,t])}{\E{\cp(W_1[0,t])}}\, \le \, \frac{\cp(W_1[0,a_{n+1}])}{\E{\cp(W_1[0,a_n])}}.
\end{equation}
Moreover, applying Proposition~\ref{prop-expected} again gives 
$$
\mathbb E[\cp(W_1[a_n,a_{n+1}])]\ = \mathbb E[\cp(W_1[0,a_{n+1}-a_n])]=
\mathcal O\left(\ \frac{a_{n+1}-a_n}{\log (a_{n+1}-a_n)}\ \right)=
\ o\left(\frac{a_n}{\log a_n}\right).
$$
Then using that for any sets $A$ and $B$, one has 
$\cp(A)\le \cp(A\cup B)\le \cp(A)+\cp(B)$, we deduce that 
\[
\lim_{n\to\infty} \frac{\E{\cp(W_1[0,a_{n+1}])}}{\E{\cp(W_1[0,a_n])}}=1, 
\]
which, together with \eqref{asconv1} and \eqref{asconv}, proves the almost sure convergence.

The convergence in $L^p$ follows from the boundedness result proved in Section~\ref{Sec.updev}, see Remark \ref{rem.Lp}.   \hfill $\square$

\vspace{0.2cm}
Finally we note that the bound on the variance \eqref{ineq-var} follows from \reff{legall-2}, \reff{chi.ell}, \reff{cheb-2} and \reff{cheb-2bis}.

\bibliographystyle{abbrv}
\bibliography{biblio}

\end{document}